\documentclass[a4paper, final, sn-mathphys-num]{sn-jnl}
\usepackage{amsfonts, amssymb, amsthm, amsmath}
\usepackage[utf8]{inputenc}
\usepackage[british]{babel}
\usepackage{graphicx}
\usepackage[algoruled]{algorithm2e}
\usepackage{booktabs, multirow}

\usepackage{float}

\newcommand{\fmin}[1]{f_{min}#1}
\newcommand{\imin}[1]{\mathcal{I}_{min}#1}
\newcommand{\norm}[1]{\left\lVert#1\right\lVert}
\newcommand{\abs}[1]{\left\lvert#1\right\lvert}

\newcommand{\ssuc}{\overline{\mathcal{S}}}
\newcommand{\Deltamin}{\Delta_{\min}}

\newcommand{\succe}{\overline{\mathcal{S}}_{\epsilon}}
\newcommand{\succer}{\succe^{\mathcal{R}}}
\newcommand{\succenr}{\succe^{\mathcal{NR}}}

\newcommand{\unsucce}{\mathcal{U}_{\epsilon}}
\newcommand{\accer}{\mathcal{A}_{\epsilon}^{\mathcal{R}}}
\newcommand{\accenr}{\mathcal{A}_{\epsilon}^{\mathcal{NR}}}

\newtheorem{theorem}{Theorem}[section]
\newtheorem{lemma}[theorem]{Lemma}
\newtheorem{corollary}[theorem]{Corollary}

\newtheorem{assumption}[theorem]{Assumption}
\newtheorem{remark}[theorem]{Remark}
\theoremstyle{definition}\newtheorem{defn}[theorem]{Definition}

\usepackage[normalem]{ulem}

\usepackage[notref, notcite]{showkeys}

\title[A DF TR approach for LOVO problems]{A derivative-free trust-region approach for Low Order-Value Optimization problems\footnote{This study was financed in part by the Coordenação de Aperfeiçoamento de Pessoal de Nível Superior – Brasil (CAPES) – Finance Code 001}} %
\author[1]{ %
    \fnm{Anderson E.} \sur{Schwertner} \orcid{https://orcid.org/0000-0003-4660-4611}} \email{schwertner.ae@gmail.com} \equalcont{These authors contributed equally to this work.} %
\author*[2]{
    \fnm{Francisco N. C.} \sur{Sobral} \orcid{https://orcid.org/0000-0003-4963-0946}}\email{fncsobral@uem.br} \equalcont{These authors contributed equally to this work.}%

\affil[1]{Graduate Program in Mathematics, State University of Maringá, Paraná, Brazil}
\affil*[2]{Department of Mathematics, State University of Maringá, Paraná, Brazil}

\begin{document}
	
\abstract{The Low Order-Value Optimization (LOVO) problem involves minimizing the minimum among a finite number of function values within a feasible set. LOVO has several practical applications such as robust parameter estimation, protein alignment, portfolio optimization, among others. In this work, we are interested in the constrained nonlinear optimization LOVO problem of minimizing the minimum between a finite number of function values subject to a nonempty closed convex set where each function is black-box and continuously differentiable, but the derivatives are not available. We develop the first derivative-free trust-region algorithm for constrained LOVO problems with convergence to weakly critical points. Under suitable conditions, we establish the global convergence of the algorithm and also its worst-case iteration complexity analysis. An initial open-source implementation using only linear interpolation models is developed. Extensive numerical experiments and comparison with existing alternatives show the properties and the efficiency of the proposed approach when solving LOVO problems.}


\maketitle

\section{Introduction}\label{sec:introduction}

Let us consider the following constrained nonlinear optimization problem
\begin{equation}\label{eq_lovo_sem_derivadas}
    \min_{x \in \Omega} \fmin{(x)} = \min_{x \in \Omega} \min\{f_{1}(x), \dots,f_{r}(x)\},
\end{equation}
where $\Omega \subset \mathbb{R}^{n}$ is a nonempty closed convex set and $f_{i}: \mathbb{R}^{n} \to \mathbb{R}$, $i = 1, \dots, r$, are continuously differentiable black-box functions. Although we consider that each function that composes the objective function is smooth, we assume that there is no information available about the first and second-order derivatives. 

Problem (\ref{eq_lovo_sem_derivadas}) is known as the Low Order-Value Optimization problem (LOVO), for which extensive theory and application exist \cite{AMMY2009, M2009}. Applications include protein alignment \cite{AMM2008, AMMY2009, M2012, MAM2007}, robust parameter estimation \cite{AMMY2009, CLSS2021}, portfolio optimization \cite{BBKM2011, M2012}, robust response surface methodology \cite{PFF2021}, among others. In recent years, several authors have developed LOVO versions of classical continuous optimization methods, such as the Levenberg-Marquardt \cite{CLSS2021}, Augmented Lagrangian \cite{AMMY2009}, Trust-Region \cite{AMM2008} and Coordinate-Search \cite{M2012} algorithms.

In this article, we address derivative-free optimization, which has several scientific, medical, social, and engineering applications \cite{AH2017, CST1998, CSV2009, LMW2019}, specially in situations where the objective function is black-box, derived from a simulation, or without a known analytical expression. Detailed surveys of derivative-free optimization methods can be found in \cite{LMW2019, RS2013}, and theoretical aspects can be revisited in \cite{AH2017, CSV2009}. 

Among the various algorithms developed for derivative-free optimization, variations of the trust-region method stand out. The study of such methods began with Winfield \cite{W1969, W1973} and became popular due to the good numerical performance of the algorithms proposed by Powell \cite{P2002, P2006, P2009}. There are several derivative-free trust-region algorithms described in the literature, which can be applied to unconstrained problems \cite{CST1997, CSV2009a, FMN2009, P2006, WRS2008}, bound-constrained \cite{GTT2011, P2009, W2008}, linearly-constrained \cite{GHA2014, P2015}, convex-constrained \cite{CKPRS2013, HR2021, VKPS2017}, and general-constrained \cite{BFMS2013, CKP2015, FKSS2017, TIS2021} problems, as well as composite nonsmooth optimization \cite{GJV2016, GYY2016, LMZ2021}, and multi-objective optimization \cite{TE2019} problems, among others.

The goal of this work is to develop and analyse a derivative-free optimization algorithm designed for convex-constrained LOVO problems. Note that the LOVO problem can be understood as a particular case of nonsmooth composite optimization. Using this observation, \cite{GJV2016} and~\cite{GYY2016} proposed model-based derivative-free algorithms for objective functions of the form $f + h \circ g$, where $f$ and $g$ are smooth, and $h$ is convex. However, such methods are not compatible with the specificities of the LOVO problem since the minimum function is not convex. Furthermore, the construction of the models involves the evaluation of $h \circ g$ at each one of the points in the sample set, which translates into higher computational costs when compared to the strategy adopted in our approach.

In~\cite{LMZ2021} it was proposed a manifold sampling algorithm for minimizing nonsmooth compositions $g \circ f$, where $f : \mathbb{R}^n \to \mathbb{R}^r$ is continuously differentiable, with non-available Jacobian, and $g : \mathbb{R}^r \to \mathbb{R}$ is a continuous selection. A \emph{continuous selection} is defined as a continuous function $g: \mathbb{R}^r \to \mathbb{R}$, such that $g(z) \in \{g_i(z), g_i \in \mathcal{G}\}$, where $g_i : \mathbb{R}^r \to \mathbb{R}$ and $\mathcal{G}$ is a finite set of selection functions. In our context, $f_{min}$ can be written as $g \circ f$, if we define $f(x) = (f_1(x), \dots, f_r(x))$ and $g(z) = \min \{g_1(z), \dots, g_r(z)\}$, where $g_i(z) = z_i$. From a theoretical point of view, the proposed algorithm generates a sequence of iterates that converge to a Clarke stationary point. However, in~\cite{M2012} it is shown that the concept of Clarke stationarity is not sufficiently satisfactory for LOVO problems, which makes it necessary to adopt methods that provably converge to points that satisfy stronger optimality conditions, such as \cite[Theorem 6.1]{M2012}.

The authors in~\cite{HR2021} and~\cite{Roberts2025} define a suitable notion of \(\Lambda\)-poised interpolation models and show convergence and worst-case iteration complexity results for general convex-constrained derivative-free optimization problems. The algorithm is strongly based in the method proposed in~\cite{CKPRS2013}. Recently, convergence and worst-case iteration complexity was also discussed in~\cite{Alvarez2025} in the context of convex-constrained LOVO problems based on derivatives.

The main contributions of the present work can be stated as follows
\begin{itemize}
    \item We develop a trust-region derivative-free algorithm for convex-constrained LOVO problems, which is suitable when the projection onto the convex set is available. This in turn, is an extension of~\cite{AMM2008} to the derivative-free constrained case;
    \item Inexact interpolation models are allowed, as well as the use of two different radii~\cite{P2006}, improving the results of~\cite{VKPS2017};
    \item Convergence theory and worst-case iteration complexity are provided, showing that the behaviour of the algorithm follows the usual trust-region framework;
    \item An efficient and open-source numerical implementation is provided and is competitive against algorithms for nonsmooth derivative-free problems.
\end{itemize}

This work is organized as follows. In Section \ref{sec:df_algorithm}, we introduce our derivative-free trust-region algorithm for LOVO problems. In Section \ref{sec:convergence}, we discuss its convergence analysis and global convergence results. In Section \ref{sec:complexity}, we study the worst-case complexity of our algorithm. Implementation details and numerical experiments are discussed in Section~\ref{sec:experiments}, and conclusions are given in Section \ref{sec:conclusions}.

\paragraph{Notation} $\norm{\cdot}$ denotes the Euclidean norm of vectors and matrices; $\mathcal{P}_{2}^{n}$ is the set of all polynomials of degree less than or equal to $2$ in $\mathbb{R}^{n}$; and $\overline{B}(x^{*}, \delta) = \{ x \in \mathbb{R}^{n} \text{ } | \text{ } \norm{x - x^{*}} \leq \delta \}$.

\section{A derivative-free trust-region LOVO algorithm}\label{sec:df_algorithm}

In order to help us with the theoretical results exposed in this text, consider the following definition.

\begin{defn}\cite[Definition 1]{CLSS2021}
Given a feasible point $x \in \Omega$ we define \(\mathcal{I} = \{1, \dots, r\}\) and \(\imin (x) = \{ i \in \{1, \dots , r\} \hspace{2pt} | \hspace{2pt} f_{i}(x) = \fmin{(x)} \}\).
\end{defn}

We consider the following assumptions on the objective function.

\begin{assumption}\label{hip_f_diferenciavel}
For every $i \in \mathcal{I}$, the function $f_{i}: \mathbb{R}^{n} \to \mathbb{R}$ is continuously differentiable and its gradient $\nabla f_{i}: \mathbb{R}^{n} \to \mathbb{R}^{n}$ is Lipschitz continuous, with constant $L_{i} > 0$, in a sufficiently large open bounded domain $\mathcal{X} \subset \mathbb{R}^{n}$.
\end{assumption}

\begin{assumption}\label{hip_f_limitada}
For every $i \in \mathcal{I}$, the function $f_{i}: \mathbb{R}^{n} \to \mathbb{R}$ is bounded below in $\Omega \subset \mathbb{R}^{n}$, i.e., there is a constant $M_{i} \in \mathbb{R}$ such that $f_{i}(x) \geq M_{i}$, for all $x \in \Omega$.
\end{assumption}

Note that if Assumption \ref{hip_f_limitada} holds and $M := \min_{i \in \mathcal{I}}\{M_{i}\}$, thus the function $\fmin$ is bounded below in $\Omega$ by the constant $M$.

\subsection{Trust-region framework}\label{subsec:tr_framework}

Our algorithm is based on the trust-region frameworks of \cite{AMM2008, VKPS2017}. At each iteration $k \in \mathbb{N}$, we consider the current iterate $x_{k} \in \Omega$, an associated index $i_{k} \in \imin(x_{k})$ and the quadratic model for $f_{i_{k}}$
\begin{equation}\label{eq_def_modelo}
    \mathrm{m}_{k}(x) = \mathrm{m}_{k}(x_{k} + d) = b_{k} + \mathbf{g}_{k}^{T} d + \frac{1}{2}d^{T}\mathbf{H}_{k}d,
\end{equation}
where $d = x - x_{k} \in \mathbb{R}^{n}$, $b_{k} \in \mathbb{R}$, $\mathbf{g}_{k} \in \mathbb{R}^{n}$ and $\mathbf{H}_{k} \in \mathbb{R}^{n \times n}$ is a symmetric matrix. In \cite{AMM2008}, a similar approach is taken with $\mathbf{g}_{k} = \nabla f_{i_{k}}(x_{k})$ and $\mathbf{H}_{k} = \nabla^{2}f_{i_{k}}(x_{k})$. Note that only one index $i_{k} \in \imin{(x_{k})}$ is chosen at iteration $k \in \mathbb{N}$. Thus, we will omit the indication of such index in the model definition and other expressions that depend on the choice of index.


We also follow the practical ideas of Powell \cite{P2006, P2009}, further developed in \cite{VKPS2017}, of using two radii: $\delta_{k}$ is related to the quality of the model, and $\Delta_{k}$ is associated with the trust-region. Assumption \ref{hip_grad_limitante} defines what ``quality of the model'' means and is an assumption closely connected to the ``fully linear models'' of \cite{CSV2009}.

\begin{assumption}\label{hip_grad_limitante}
  For all $k \in \mathbb{N}$, $\norm{\nabla f_{i_{k}}(x) - \nabla\mathrm{m}_{k}(x)} \leq \kappa_{g}\delta_{k}$, for all $x \in \mathcal{X} \cap \overline{B}(x_{k}, \delta_{k})$ and some constant $\kappa_{g} > 0$ independent of $x$ and $k$.
\end{assumption}

We define the stationarity measure at $x_{k}$ for the problem of minimizing $\mathrm{m}_{k}$ over the set $\Omega$ by
\begin{equation}\label{eq_def_projecao}
    \pi_{k} = \norm{P_{\Omega}(x_{k} - \mathbf{g}_{k}) - x_{k}},
\end{equation}
where $P_{\Omega}$ denotes the orthogonal projection onto $\Omega$, which exists because it is a closed convex set \cite{CGT2000}. In our context, given $i \in \mathcal{I}$, we say that a point $x_{*} \in \Omega$ is stationary for the problem $\min_{x \in \Omega} f_{i}(x)$ when $\norm{P_{\Omega}(x_{*} - \nabla f_{i}(x_{*})) - x_{*}} = 0$ \cite[Section 12.1.4]{CGT2000}.

At iteration $k \in \mathbb{N}$, the candidate point $x_{k} + d_{k} \in \Omega$ is computed, where $d_{k}$ is the approximate solution of the following trust-region subproblem
\begin{equation}\label{eq_def_subproblema}
    \begin{array}{cc}
        \text{minimize} & \mathrm{m}_{k}(x_{k} + d) \\
        \text{subject to} & x_{k} + d \in \Omega \\
                         & \norm{d} \leq \Delta_{k}.
    \end{array}
\end{equation}
If the model is accurate for $f_{i_{k}}$, we expect that $x_{k} + d_{k}$ may also decrease $\fmin$, since $x_{k} + d_{k} \in \Omega$ and
$
\fmin(x_{k} + d_{k}) \leq f_{i_{k}}(x_{k} + d_{k}) \leq f_{i_{k}}(x_{k}).
$
By ``approximate solution'' we mean that $x_{k} + d_{k}$ must satisfy a sufficient decrease condition of $\mathrm{m}_{k}$, given by Assumption \ref{hip_d_decrescimo}.


\begin{assumption}\label{hip_d_decrescimo}
  The approximate solution of~\eqref{eq_def_subproblema}
  satisfies the sufficient decrease condition
\begin{equation*}
    \mathrm{m}_{k}(x_{k}) - \mathrm{m}_{k}(x_{k} + d_{k}) \geq \theta\pi_{k}\min\bigg\{ \frac{\pi_{k}}{1 + \norm{\mathbf{H}_{k}}}, \Delta_{k}, 1 \bigg\},
\end{equation*}
for some constant $\theta > 0$ independent of $k$.
\end{assumption}

Conditions of this type are well-known in trust-region strategies, and were used in different contexts by several authors, as we can see in \cite{CKPRS2013, VKPS2017}, \cite[Lemma 11.3]{AH2017}, \cite[Assumption AA.1]{CGT2000}, \cite[Theorem 10.1]{CSV2009}, \cite[Assumption 3.3]{HR2021}, \cite[Lemma 4.5]{NW2006}, and \cite[Section 3.1.4]{T2011}.

We accept step $x_{k} + d_{k}$ when the ratio between actual and predicted reductions 
\begin{equation}\label{eq_def_coef_reducao_relativa}
    \rho_{k} = \frac{\fmin (x_{k}) - \fmin (x_{k} + d_{k})}{\mathrm{m}_{k}(x_{k}) - \mathrm{m}_{k}(x_{k} + d_{k})}
\end{equation}
is greater than or equal to a fixed constant $\eta > 0$. In this case, the new iterate becomes $x_{k + 1} = x_{k} + d_{k}$, the model is updated and the trust-region radius $\Delta_{k}$ is
possibly increased. Otherwise, we reject the step and $\Delta_{k}$ is decreased.

 In the remainder of this work, unless otherwise stated, we will assume that Assumptions~\ref{hip_f_diferenciavel}, \ref{hip_f_limitada}, \ref{hip_grad_limitante} and~\ref{hip_d_decrescimo} hold.

\subsection{Algorithm}

In this section, we present the derivative-free trust-region algorithm for solving the LOVO problem (\ref{eq_lovo_sem_derivadas}), without specification about the model update or the solution of subproblem (\ref{eq_def_subproblema}). Our algorithm is based on the ideas discussed in \cite{AMM2008} and on the algorithms proposed by \cite{CKPRS2013} and \cite{VKPS2017}.

\begin{algorithm}[H]\footnotesize
    \LinesNumbered
    \caption{\textsc{Derivative-free trust-region LOVO algorithm}}\label{alg_lowder}
    \KwIn{$x_{0} \in \Omega$, $\beta > 0$, $0 < \delta_{0} \leq \Delta_{0}$, $0 < \tau_{1} \leq \tau_{2} < 1 \leq \tau_{3} \leq \tau_{4}$, $\eta_{1} \in (0, 1)$, $0 \leq \eta < \eta_{1} \leq \eta_{2}$, $\Gamma = 0$, $1 \leq \Gamma_{\max} \in \mathbb{N}$, $i_{0} \in \imin{(x_{0})}$.}
    \For{$k = 0, 1, \dots$}{
        Construct the model $\mathrm{m}_{k}$. \\
        \eIf(\nllabel{alg:criticality}\tcc*[f]{Criticality Phase}){$\delta_{k} > \beta \pi_{k}$}{
            Let $x_{k + 1} = x_{k}$, $i_{k + 1} = i_{k}$, $\rho_{k} = 0$, $d_{k} = 0$, $\delta_{k + 1} = \tau_{1} \delta_{k}$, and choose $\Delta_{k + 1} \in [\tau_{1}\Delta_{k}, \tau_{2}\Delta_{k}]$.
            }{
            Find an approximate solution $d_{k}$ for (\ref{eq_def_subproblema}) that satisfies Assumption \ref{hip_d_decrescimo}. \\
            Compute $\rho_{k}$ by (\ref{eq_def_coef_reducao_relativa}). \\
            \eIf(\nllabel{alg:acceptance}\tcc*[f]{Step Acceptance Phase}){$\rho_{k} \geq \eta$}{
                Let $x_{k + 1} = x_{k} + d_{k}$, and choose $i_{k+1} \in \imin (x_{k+1})$.
                }{
                Let $x_{k + 1} = x_{k}$, and $i_{k + 1} = i_{k}$.
                }
            \bf{if} $\rho_{k} \geq \eta_{1}$ \bf{then} $\Gamma = 0$ \bf{end} \\ 
            \eIf(\nllabel{alg:adjustment}\tcc*[f]{Radii Adjustment Phase}){$\rho_{k} \geq \eta$ \bf{and} $i_{k+1} \neq i_{k}$ \bf{and} $\Gamma \leq \Gamma_{\max}$}{
                Let $\delta_{k+1} = \tau_{4}\delta_{k}$, $\Delta_{k+1} = \tau_{4}\Delta_{k}$, and $\Gamma = \Gamma + 1$.\nllabel{alg:adjustment_increase}
                }(\nllabel{alg:update}\tcc*[f]{Radii Update Phase}){
                Let
                \begin{equation*}
                \delta_{k + 1} = 
                    \begin{cases}
                        \tau_{1}\delta_{k} \text{, if } \rho_{k} < \eta_{1}; \\
                        \tau_{3}\delta_{k} \text{, if } \rho_{k} > \eta_{2} \text{ and } \norm{d_{k}} = \Delta_{k}; \\
                        \delta_{k} \text{, otherwise;} \\
                    \end{cases}
                \end{equation*}
                and
                \begin{equation*}
                \Delta_{k + 1} = 
                    \begin{cases}
                        \tau_{1}\Delta_{k} \text{, if } \rho_{k} < \eta_{1}; \\
                        \tau_{3}\Delta_{k} \text{, if } \rho_{k} > \eta_{2} \text{ and } \norm{d_{k}} = \Delta_{k}; \\
                        \Delta_{k} \text{, otherwise.}
                    \end{cases}
                \end{equation*}
                }
            }
    }
\end{algorithm}

The general idea of the algorithm is given below:
\begin{itemize}
\item The algorithm allows freedom in the choice of
    $\mathrm{m}_{k}$ as long as it satisfies
    Assumption~\ref{hip_grad_limitante}. However, its practical
    efficiency greatly relies on the way it is implemented, as
    discussed in Section~\ref{sec:experiments}.
    \item When $\delta_{k}$ is large with respect to $\pi_{k}$, we cannot guarantee that the model accurately represents $f_{i_{k}}$. Hence, if $\delta_{k} > \beta \pi_{k}$, the radius $\delta_{k}$ is reduced in an attempt to find a more accurate model.
    \item When a successful iteration occurs ($\rho_{k} \geq \eta$) and another function is taken as the representative of $\fmin$ ($i_{k+1} \neq i_{k}$), we increase $\delta_{k+1}$ and $\Delta_{k+1}$ by a factor of $\tau_{4} \geq 1$.
      After a successful iteration with $\rho_{k} \geq \eta_{1}$ and index swap, this procedure is executed at most $\Gamma_{\max} \in \mathbb{N}$ iterations with index swap satisfying $\eta \leq \rho_{k} < \eta_{1}$. Thus, the radii of the other iterations satisfying $\rho_{k} < \eta_{1}$ continue to be reduced by $\tau_{1}$. By allowing the radius of the trust-region to increase, we expect the algorithm to try larger steps when a new $f_{i_{k + 1}}$ is considered.
      In~\cite{AMM2008}, $\Delta_{k+1}$ is reset to $\Delta_{0}$
      whenever index swap occurs, but we observed that such choice
      results in worse complexity results.
    \item When $\abs{\mathcal{I}} = 1$ (usual convex constrained optimization problem), Algorithm \ref{alg_lowder} reduces to the derivative-free trust-region algorithm proposed by~\cite{VKPS2017}.
\end{itemize}

\section{Convergence analysis}\label{sec:convergence}

In this section, we will provide global convergence results, in the
LOVO sense, for sequences generated by Algorithm~\ref{alg_lowder}. We
will show that every point of accumulation of a sequence generated by
Algorithm \ref{alg_lowder} is stationary for some index
$i \in \mathcal{I}$.

We start by recalling the \emph{non-expansive property} of the orthogonal projection operator $P_\Omega$, which states that $\left\|P_\Omega(x) - P_\Omega(y)\right\| \le \left\|x - y\right\|$ for all $x, y \in \mathbb{R}^n$. This property will be used for some convergence results.

In the following lemma, we prove that the Hessian of the model $\mathrm{m}_{k}$ is bounded.

\begin{lemma}\label{lema_hessiana_limitada}
Suppose that Assumptions \ref{hip_f_diferenciavel} and \ref{hip_grad_limitante} hold. Then, for every iteration $k \in \mathbb{N}$,
\begin{equation*}
    \norm{\mathbf{H}_{k}} \leq \kappa_{H} - 1,
\end{equation*}
where $\kappa_{H} := 2\kappa_{g} + L + 1$ and $L := \displaystyle\max_{i \in \mathcal{I}} \{L_{i} \}$.
\end{lemma}
\begin{proof}
Let $k \in \mathbb{N}$, $i_{k} \in \imin{(x_{k})}$ be the chosen index associated with $\mathrm{m}_{k}$ in this iteration and $d \in \mathbb{R}^{n}$ an arbitrary direction satisfying $\norm{d} = \delta_{k}$. Initially, note that, $\nabla \mathrm{m}_{k}(x_{k} + d) = \mathbf{g}_{k} + \mathbf{H}_{k}d$, and $\nabla \mathrm{m}_{k}(x_{k}) = \mathbf{g}_{k}$. The results follow by \cite[Lemma 1]{VKPS2017} by observing that $L_{i_{k}} \leq L$.
\end{proof}

Analogously to \cite{CKPRS2013, VKPS2017}, let us consider the following sets of indexes
\begin{equation}\label{eq_def_conjunto_sucesso}
    \mathcal{S} = \{k \in \mathbb{N} \text{ } | \text{ } \rho_{k} \geq \eta \} \text{ and } \overline{\mathcal{S}} = \{k \in \mathbb{N} \text{ } | \text{ } \rho_{k} \geq \eta_{1} \},
\end{equation}
and note that $\overline{\mathcal{S}} \subset \mathcal{S}$. The next lemma establishes that if the trust-region radius is sufficiently small, then the iteration will be successful.

\begin{lemma}\label{lema_Delta_sucesso}
Suppose that Assumptions \ref{hip_f_diferenciavel},  \ref{hip_grad_limitante} and \ref{hip_d_decrescimo} hold. Consider the set 
\begin{equation}\label{lema_Delta_sucesso_eq_01}
    \mathcal{K} = \bigg\{k \in \mathbb{N}: \Delta_{k} \leq \min \bigg\{\frac{\pi_{k}}{\kappa_{H}}, \frac{(1 - \eta_{1})\pi_{k}}{c_{1}}, \beta \pi_{k}, 1 \bigg\} \bigg\},
\end{equation}
where $c_{1} := \big(L + \kappa_{g} + \kappa_{H} / 2 \big) / \theta$. If $k \in \mathcal{K}$, then $k \in \overline{\mathcal{S}}$.
\end{lemma}
\begin{proof}
Let $k \in \mathcal{K}$ and consider $i_{k} \in \imin{(x_{k})}$ the index associated with the model $\mathrm{m}_{k}$ in this iteration. By the Mean Value Theorem, there exists $\xi_{k} \in (0, 1)$ such that
\begin{equation*}
    f_{i_{k}}(x_{k} + d_{k}) = f_{i_{k}}(x_{k}) + \nabla f_{i_{k}}(x_{k} + \xi_{k}d_{k})^{T}d_{k}.
\end{equation*}
Therefore, by Assumptions \ref{hip_f_diferenciavel}, \ref{hip_grad_limitante} and Lemma \ref{lema_hessiana_limitada},
by following the same arguments as \cite[Lemma 3.1]{CKPRS2013}, we obtain
\begin{equation}\label{lema_Delta_sucesso_eq_02}
    \abs{\mathrm{m}_{k}(x_{k}) - \mathrm{m}_{k}(x_{k} + d_{k}) + f_{i_{k}}(x_{k} + d_{k}) - f_{i_{k}}(x_{k})} \leq \left(L + \kappa_{g} + \frac{1}{2}\kappa_{H} \right) \Delta_{k}^{2}
\end{equation}

By \eqref{lema_Delta_sucesso_eq_01} we have that $\pi_{k} > 0$.
Thus, the sufficient decrease condition of Assumption \ref{hip_d_decrescimo} implies
$
    \mathrm{m}_{k}(x_{k}) - \mathrm{m}_{k}(x_{k} + d_{k}) > 0
$.
The key point in this proof is to observe that, since $i_{k} \in \imin{(x_{k})}$, we know that $f_{i_{k}}(x_{k}) = \fmin{(x_{k})}$ and $\fmin{(x_{k} + d_{k})} \leq f_{i_{k}}(x_{k} + d_{k})$. By (\ref{eq_def_coef_reducao_relativa}), (\ref{lema_Delta_sucesso_eq_02}), Assumption \ref{hip_d_decrescimo} and Lemma \ref{lema_hessiana_limitada},
\begin{equation*}
    \begin{aligned}
        1 - \rho_{k} 
        &= 1 - \frac{\fmin{(x_{k})} - \fmin{(x_{k} + d_{k})}}{\mathrm{m}_{k}(x_{k}) - \mathrm{m}_{k}(x_{k} + d_{k})} \\
        &\leq 1 - \frac{f_{i_{k}}{(x_{k})} - f_{i_{k}}{(x_{k} + d_{k})}}{\mathrm{m}_{k}(x_{k}) - \mathrm{m}_{k}(x_{k} + d_{k})} \\
        &\leq \frac{\abs{\mathrm{m}_{k}(x_{k}) - \mathrm{m}_{k}(x_{k} + d_{k}) - f_{i_{k}}(x_{k}) + f_{i_{k}}(x_{k} + d_{k}) }}{\mathrm{m}_{k}(x_{k}) - \mathrm{m}_{k}(x_{k} + d_{k})} \\
        &\leq \frac{\big(L + \kappa_{g} + \frac{1}{2}\kappa_{H} \big)\Delta_{k}^{2}}{\theta\pi_{k}\min\Big\{ \frac{\pi_{k}}{\kappa_{H}}, \Delta_{k}, 1 \Big\}} \\
        &= \frac{c_{1}\Delta_{k}^{2}}{\pi_{k}\min\Big\{ \frac{\pi_{k}}{\kappa_{H}}, \Delta_{k}, 1 \Big\}},
    \end{aligned}
\end{equation*}
where $c_{1} = \left( L + \kappa_{g} + \kappa_{H} / 2 \right) / \theta$. The remaining of the proof follows exactly as~\cite[Lemma~3.1]{CKPRS2013}, adapting our constants.
\end{proof}

The following result is a direct consequence of Lemma \ref{lema_Delta_sucesso}. Assuming that $\pi_{k} \geq \varepsilon > 0$, we can prove that the trust-region radius is bounded below by $\Deltamin$, which depends on the parameters of the algorithm, the characteristics of problem (\ref{eq_lovo_sem_derivadas}) and the type of model $\mathrm{m}_{k}$ used.
\begin{corollary}\label{corolario_Delta_min}
Suppose that Assumptions \ref{hip_f_diferenciavel},  \ref{hip_grad_limitante} and \ref{hip_d_decrescimo} hold, and let $\varepsilon > 0$. If $\pi_{k} \geq \varepsilon$, then
\begin{equation*} 
    \Delta_{k} \geq \Deltamin := \min\left\{ \Delta_{0}, \tau_{1} \min \left\{\frac{\varepsilon}{\kappa_{H}}, \frac{(1 - \eta_{1})\varepsilon}{c_{1}}, \beta \varepsilon, 1 \right\} \right\}.
\end{equation*}
\end{corollary}

As mentioned earlier, the radius of the sample region $\delta_{k}$ controls the quality of the model. More specifically, Assumption \ref{hip_grad_limitante} tells us that the smaller the sample radius, the better the models represent $f_{i_{k}}$. Therefore, it is reasonable to expect $\delta_{k}$ to go to zero. The following two lemmas show us that this happens.

\begin{lemma}\label{lema_infinitos_sucessos}
If $\abs{\overline{\mathcal{S}}} = \infty$, then the sequence $\{\delta_{k}\}_{k \in \overline{\mathcal{S}}}$ converges to zero.
\end{lemma}
\begin{proof}
The proof follows the same arguments as \cite[Lemma 3]{VKPS2017}. Note that $\{ \fmin{(x_{k})} \}_{k \in \mathbb{N}}$ is a monotone nonincreasing sequence and bounded below, and the functions $f_{i}$, $i = 1, \dots, r$, are bounded below in $\Omega$.
\end{proof}

\begin{lemma}\label{lema_delta_converge_zero}
The sequence $\{\delta_{k}\}_{k \in \mathbb{N}}$ converges to zero.
\end{lemma}
\begin{proof}
The proof is based on \cite[Lemma 3]{VKPS2017}, with the main difference being due to the conditions at line~\ref{alg:adjustment} in Algorithm~\ref{alg_lowder}. We will split the demonstration into two cases.

Case 1: $\overline{\mathcal{S}}$ is finite. Initially consider that the set $\overline{\mathcal{S}}$ is finite. Hence, there is a finite number of iterations for which the radii adjustment phase (line \ref{alg:adjustment}) is called. In fact, such phase is called for at most $\Gamma_{\max}\abs{\overline{\mathcal{S}}}$ iterations. 
Thus, there is $k_{0} \in \mathbb{N}$ such that for every $k \geq k_{0}$, the iteration $k \notin \overline{\mathcal{S}}$ and the radii adjustment phase (line \ref{alg:adjustment}) is not called. Let $\mathcal{M}_{1} := \{k \in \mathbb{N}\text{ }|\text{ } k \geq k_{0}\}$. Note that, by the radii update phase (line \ref{alg:update}) present in the algorithm, we have that for every iteration $k \in \mathcal{M}_{1}$ $\delta_{k + 1} = \tau_{1}\delta_ {k}$, where $\tau_{1} < 1$. Therefore, given that $\delta_{k_{0}}$
%
is a constant, it follows that
\begin{equation*}
    \lim_{k \rightarrow \infty} \delta_{k} = \lim_{k \in \mathcal{M}_{1}} \delta_{k} = \lim_{k \rightarrow \infty} \tau_{1}^{k}\delta_{k_{0}} = 0.
\end{equation*}

Case 2: $\overline{\mathcal{S}}$ is infinite. Now, assume that the set $\overline{\mathcal{S}}$ is infinite and let $\mathcal{M}_{2}:= \{ k \in \mathbb{N}\text{ }|\text{ } k \notin \overline{\mathcal{S}} \}$. If the set $\mathcal{M}_{2}$ is finite, then by Lemma \ref{lema_infinitos_sucessos} it follows that
\begin{equation*}
    \lim_{k \to \infty} \delta_{k} = \lim_{k \in \overline{\mathcal{S}}} \delta_{k} = 0.
\end{equation*}

This leaves only the case where $\mathcal{M}_{2}$ is infinite. Note that, by Lemma \ref{lema_infinitos_sucessos}, we obtain the desired limit for indices in $\overline{\mathcal{S}}$. We will extend this result to the entire sequence. For this purpose, let $k \in \mathcal{M}_{2}$ and $l_{k}$ be the last iteration in $\ssuc$ before $k$. Since between the iterations $l_{k}$ and $k$ the radii adjustment phase (line \ref{alg:adjustment}) is called at most $\Gamma_{\max}$ times and $\tau_{4} \geq \tau_{3} \geq 1$, we have that the value of $\delta_{k}$ is at most $\tau_{4}^{\Gamma_{\max}}\delta_{l_{k}}$. Therefore, 
\begin{equation*}
    \lim_{k \in \mathcal{M}_{2}} \delta_{k} \leq \lim_{k \in \mathcal{M}_{2}} \tau_{4}^{\Gamma_{\max}}\delta_{l_{k}} = \tau_{4}^{\Gamma_{\max}}\lim_{l_{k} \in \overline{\mathcal{S}}} \delta_{l_{k}} = 0,
\end{equation*}
and we conclude the proof.
\end{proof}

Next, we show a weak convergence result for the problem of minimizing the model $\mathrm{m}_{k}$ in the feasible region $\Omega$.

\begin{lemma}\label{lema_pi_sub_converge_zero}
The sequence $\{\pi_{k}\}_{k \in \mathbb{N}}$ admits a subsequence that converges to zero. 
\end{lemma}
\begin{proof}
We will show that
\begin{equation*}
    \liminf_{k \rightarrow \infty}\pi_{k} = 0.
\end{equation*}

In fact, suppose by contradiction that there are $\varepsilon > 0$ and an integer $k_{0} > 0$ such that $\pi_{k} \geq \varepsilon$ for all $k \geq k_{0}$. 
Let $k \in \overline{\mathcal{S}}$, with $k \geq k_{0}$ arbitrary. By the definition of $\rho_{k}$ given in (\ref{eq_def_coef_reducao_relativa}), Assumption \ref{hip_d_decrescimo}, the contradiction hypothesis and
Corollary~\ref{corolario_Delta_min}, it follows that 
\begin{equation*}
    \begin{aligned}
        \fmin{(x_{k})} - \fmin{(x_{k} + d_{k})}
        &\geq \rho_{k} \theta\pi_{k}\min\Big\{ \frac{\pi_{k}}{\kappa_{H}}, \Delta_{k}, 1 \Big\} \\
        &\geq \eta_{1}\theta\varepsilon\min\Big\{ \frac{\varepsilon}{\kappa_{H}}, \Deltamin, 1 \Big\}.
    \end{aligned}
\end{equation*}

By Assumption \ref{hip_f_limitada}, $\{\fmin{(x_{k})} \}_{k \in \mathbb{N}}$ is bounded below, and given that it is a monotone nonincreasing sequence, it follows that $\fmin{(x_{k})} - \fmin{(x_{k} + d_{k})} \rightarrow 0$. Since the right-hand side of the above inequality is a positive constant, the set $\{ k \in \overline{\mathcal{S}}: k \geq k_{0}\}$ is finite, and thus the radii adjustment phase (line \ref{alg:adjustment}) is also called only a finite number of iterations. However, by the Lemma \ref{lema_delta_converge_zero} we have that $\delta_{k} \rightarrow 0$, and given that $\pi_{k} \geq \varepsilon$ for all $k \geq k_{1}$, it follows that the criticality phase (line \ref{alg:criticality}) is called only for a finite number of iterations. Thus, for every sufficiently large $k \in \mathbb{N}$, we only have iterations with $\rho_{k} < \eta_{1}$ without the radii adjustment phase (line \ref{alg:adjustment}) being called, which implies that $\Delta_{k + 1 } = \tau_{1}\Delta_{k}$. Consequently, $\Delta_{k} \rightarrow 0$, contradicting
Corollary~\ref{corolario_Delta_min}.
\end{proof}

The following result is a direct consequence of Lemma \ref{lema_pi_sub_converge_zero}. 

\begin{corollary}\label{corolario_pi_sub_converge_zero}
There is an index $i \in \mathcal{I}$, that is selected for an infinite number of iterations, such that
\begin{equation*}
    \liminf_{k \in \mathcal{J}(i)} \pi_{k} = 0,
\end{equation*}
where $\mathcal{J}(i) = \{k \in \mathbb{N}\text{ } | \text{ } i_{k} = i\}$.
\end{corollary}
\begin{proof}
  From Lemma \ref{lema_pi_sub_converge_zero}, we know that $\displaystyle\{ \pi_{k} \}_{k \in \mathbb{N}}$ admits a subsequence that converges to zero. Given that such a subsequence is infinite and that $\abs{\mathcal{I}} = r < \infty$, there
  must be an index $i \in \mathcal{I}$ that is chosen for an infinite number of iterations of that subsequence.
  Thus, collecting such iterations into the set $\mathcal{J}(i) = \{k \in \mathbb{N} \text{ }| \text{ }i_{k} = i\}$, we have that
$
    \liminf \limits_{k \in \mathcal{J}(i)} \pi_{k} = 0,
$
and we conclude the proof.
\end{proof}


If we, in addition, ask for a sufficient decrease condition $\eta > 0$ in Algorithm \ref{alg_lowder}, then we can show that $\{ \pi_{k} \}_{k \in \mathbb{N}}$ converges to zero for any subsequence when an index $i \in \mathcal{I}$ is selected for an infinite number of iterations.

\begin{lemma}\label{lema_pi_converge_zero}
Suppose that $\eta > 0$ and let $i \in \mathcal{I}$ be an index chosen for an infinite number of iterations. 
Then,
\begin{equation*}
    \lim_{k \in \mathcal{J}(i)} \pi_{k} = 0,
\end{equation*}
where $\mathcal{J}(i) = \{k \in  \mathbb{N} \text{ }|\text{ } i_{k} = i\}$.
\end{lemma}
\begin{proof}
  Initially, we know that $\mathcal{J}(i)$ exists by the same arguments of Corollary~\ref{corolario_pi_sub_converge_zero}.
  Suppose by contradiction that for some $\varepsilon > 0$ the set $\mathcal{M}_{1} = \mathcal{J}(i) \cap \{k \in \mathbb{N}: \pi_{k} \geq \varepsilon \}$ is infinite. Given $k \in \mathcal{M}_{1}$, consider $u_{k} \in \mathbb{N}$ the first iteration such that $u_{k} > k$ and $\pi_{u_{k}} \leq \varepsilon / 2$, and so $\pi_{k} - \pi_{u_{k}} \geq \varepsilon / 2$. Note that the existence of $u_{k}$ is guaranteed by Lemma \ref{lema_pi_sub_converge_zero}.
By Lemma~\ref{lema_delta_converge_zero}, there is
$k_{0} \in \mathbb{N}$ such that, for $k \geq k_{0}$, we have
$\delta_{k} \leq \varepsilon / (8 \kappa_{g})$, where $\kappa_{g}$ is
the constant given in Assumption \ref{hip_grad_limitante}. Also,
let us define
$
  \mathcal{C}_{k} = \{ j \in \mathcal{S}: k \leq j < u_{k} \}.
$
We will show that $\mathcal{C}_{k}$ is nonempty by considering two cases: $u_{k} \in \mathcal{J}(i)$ and $u_{k} \notin \mathcal{J}(i)$.


\begin{enumerate}
    \renewcommand{\theenumi}{\roman{enumi}}
    \renewcommand\labelenumi{(\theenumi)}
    \item \label{lema_pi_converge_zero_i} $u_{k}\in \mathcal{J}(i)$.

Given that $u_{k} \in \mathcal{J}(i)$, we have that $i_{k} = i_{u_{k}}$. In other words, the functions representing $\fmin$ at $k$ and $u_{k}$ are the same, let us call it $f_{i_k}$. Using the same arguments of~\cite[Lemma~5]{VKPS2017} for $f_{i_k}$, by the definition of $\pi_k$, the triangle inequality and the non-expansive property of projections, we have that $\pi_{k} - \pi_{u_{k}} \geq \varepsilon / 2 \implies 2 \|x_k - x_{u_k}\| + \|\mathbf{g}_k - \mathbf{g}_{u_k}\| \ge \varepsilon /2$. If $\mathcal{C}_k$ is empty, then $x_k = x_{u_k}$, since there are no successful iterations between $k$ and $u_k$. Therefore, using Assumption~\ref{hip_grad_limitante} and the previous inequality, we obtain $\kappa_g (\delta_k + \delta_{u_k}) \ge \varepsilon / 2$, which is a contradiction, since $\kappa_g (\delta_k + \delta_{u_k}) \le \varepsilon / 4$ for all $k \ge k_0$. We conclude that $\mathcal{C}_k$ is nonempty.

\item $u_{k}\notin \mathcal{J}(i)$. 

In this case, there is no guarantee that $\nabla f_{i_{k}}$ and $\nabla f_{i_{u_{k}}}$ are the same functions and, therefore, the argument used in item~\ref{lema_pi_converge_zero_i} is not valid. However, given that $u_{k} \notin \mathcal{J}(i)$, by the index choice condition in the step acceptance phase (line \ref{alg:acceptance}), we know that there was at least one index swap between iterations $k$ and $u_{k}$. Thus, the set $\mathcal{C}_{k}$ is nonempty.
\end{enumerate}
Therefore, in both cases we obtain that $\mathcal{C}_{k}$ is a nonempty set. Thus, by Assumption \ref{hip_d_decrescimo}, Lemma \ref{lema_hessiana_limitada}, Corollary \ref{corolario_Delta_min}, and the fact that $\pi_{j} \geq \varepsilon / 2$, for all $j \in \mathcal{C}_{k}$, we have that, for all $k \in \mathcal{M}_{1}$, $k \ge k_{0}$,

\begin{equation}\label{lema_pi_converge_zero_eq_08}
    \begin{aligned}
        \fmin{(x_{k})} - \fmin{(x_{u_{k}})}
        &\geq \sum_{j \in \mathcal{C}_{k}} \big(\fmin{(x_{j})} - \fmin{(x_{j+1})} \big) \\
        &\geq \sum_{j \in \mathcal{C}_{k}} \rho_{j} \theta \pi_{j} \min \Big\{\frac{\pi_{j}}{\kappa_{H}}, \Delta_{j}, 1 \Big\} \\
        &\geq \frac{\eta \theta \varepsilon}{2} \min \Bigg\{\frac{\varepsilon}{2\kappa_{H}}, \sum_{j \in \mathcal{C}_{k}} \Delta_{j}, 1 \Bigg\} \\
        &\geq \frac{\eta \theta \varepsilon}{2} \min \Bigg\{\frac{\varepsilon}{2\kappa_{H}}, \Deltamin, 1 \Bigg\}.
        \end{aligned}
\end{equation}

Thus, since $\eta > 0$, it follows that the right-hand side of (\ref{lema_pi_converge_zero_eq_08}) is a positive constant. On the other hand, by Assumption \ref{hip_f_limitada}, $\{\fmin{(x_{k})}\}_{k \in \mathbb{N}}$ is bounded below and, by the construction of the algorithm, is a monotone nonincreasing sequence. Hence, $\fmin{(x_{k})} - \fmin{(x_{u_{k}})} \rightarrow 0$, which is a contradiction with (\ref{lema_pi_converge_zero_eq_08}), and completes the proof.
\end{proof}

We are now ready to prove the global convergence results for Algorithm~\ref{alg_lowder}. In a way similar to the stationarity measure $\pi_{k}$ when solving subproblems~\eqref{eq_def_subproblema}, we adopt the criticality measure $\norm{ \mathcal{P}_{\Omega}(x - \nabla f(x)) - x }$, proposed by Conn, Gould and Toint \cite{CGT1988} for optimization problems involving a continuous function $f$ on a convex feasible set $\Omega$. This measure was used in \cite{CKPRS2013, CGST1996, VKPS2017}, for example, to establish global convergence results. The following theorem is the main result of this section. It allows us to further describe what kind of stationarity can be achieved by Algorithm~\ref{alg_lowder}.

\begin{theorem}\label{teorema_principal} 
  Let us define sets
  $\mathcal{J}(i) = \{k \in \mathbb{N} \mid i_{k} = i\}$, for
  $i \in \mathcal{I}$, and let
  $\{ x_{k} \}_{k \in \mathbb{N}}$ be a sequence generated by
  Algorithm \ref{alg_lowder}. The following statements hold.
\begin{enumerate}
    \renewcommand{\theenumi}{\roman{enumi}}
    \renewcommand\labelenumi{(\theenumi)}
    \item\label{t:princ:ponto1} If $\eta = 0$, then
        \begin{equation*}
            \displaystyle\liminf_{k \to \infty} \norm{\mathcal{P}_{\Omega}(x_{k} - \nabla f_{i_{k}}(x_{k})) - x_{k}} = 0.
        \end{equation*}
        Furthermore, there is an index $i \in \mathcal{I}$
        such that
        \begin{equation*}
            \displaystyle\liminf_{k \in \mathcal{J}(i)} \norm{\mathcal{P}_{\Omega}(x_{k} - \nabla f_{i_{k}}(x_{k})) - x_{k}} = 0.
        \end{equation*}
    \item\label{t:princ:ponto2} If $\eta > 0$ and $i \in \mathcal{I}$ is any index chosen for an infinite number of iterations, then
        \begin{equation*}
            \displaystyle\lim_{k \in \mathcal{J}(i)} \norm{\mathcal{P}_{\Omega}(x_{k} - \nabla f_{i_{k}}(x_{k})) - x_{k}} = 0.
        \end{equation*}
      \item\label{t:princ:ponto3} If $i \in \mathcal{I}$ is an index chosen for an infinite number of iterations and $x_{*} \in \mathbb{R}^{n}$ is an accumulation point for the subsequence $\{ x_{k} \}_{k \in \mathcal{J}(i)} \subseteq \{ x_{k} \}_{k \in \mathbb{N}}$
        then $i \in \imin{(x_{*})}$.
\end{enumerate}
\end{theorem}
\begin{proof}
Initially, note that by the triangular inequality, the non-expansive property of the orthogonal projection operator, and Assumption \ref{hip_grad_limitante}, we have
\begin{equation}\label{teorema_principal_eq_01}
    \begin{aligned}
        &\norm{P_{\Omega}(x_{k} - \nabla f_{i_{k}}(x_{k})) - x_{k}} \\
        &\hspace{1cm} = \norm{P_{\Omega}(x_{k} - \nabla f_{i_{k}}(x_{k})) - P_{\Omega}(x_{k} - \textbf{g}_{k}) + P_{\Omega}(x_{k} - \textbf{g}_{k}) - x_{k}} \\
        &\hspace{1cm} \leq \norm{P_{\Omega}(x_{k} - \nabla f_{i_{k}}(x_{k})) - P_{\Omega}(x_{k} - \textbf{g}_{k})} + \norm{P_{\Omega}(x_{k} - \textbf{g}_{k}) - x_{k}} \\
        &\hspace{1cm} \le \norm{x_{k} - \nabla f_{i_{k}}(x_{k}) - x_{k} + \textbf{g}_{k}} + \norm{P_{\Omega}(x_{k} - \textbf{g}_{k}) - x_{k}} \\
        &\hspace{1cm} = \norm{\textbf{g}_{k} - \nabla f_{i_{k}}(x_{k})} + \norm{P_{\Omega}(x_{k} - \textbf{g}_{k}) - x_{k}} \\
        &\hspace{1cm} \leq \kappa_{g}\delta_{k} + \pi_{k}.
    \end{aligned}
\end{equation}

Thus, using Lemmas~\ref{lema_delta_converge_zero} and \ref{lema_pi_sub_converge_zero}, Corollary \ref{corolario_pi_sub_converge_zero}, and~(\ref{teorema_principal_eq_01}), we obtain~\eqref{t:princ:ponto1}. On the other hand, by Lemmas \ref{lema_delta_converge_zero} and \ref{lema_pi_converge_zero}, and expression (\ref{teorema_principal_eq_01}), we conclude the proof of~\eqref{t:princ:ponto2}.

It still remains for us to prove~\eqref{t:princ:ponto3}. Note that by the statement of the theorem we have that $\lim_{k \in \mathcal{J}(i)} x_{k} = x_{*}$, and by Assumption \ref{hip_f_diferenciavel}, the functions $f_{j}$, $j \in \mathcal{I}$, are continuously differentiable. So the continuity of $f_{j}$ gives us 
\begin{equation}\label{teorema_principal_eq_02} 
    \lim_{k \in \mathcal{J}(i)} f_{j}(x_{k}) = f_{j}(x_{*}),
\end{equation}
for any index $j \in \mathcal{I}$. Given that $i \in \imin{(x_{k})}$, for every $k \in \mathcal{J}(i)$, we have by the definition of this set that $f_{i}(x_{k}) \leq f_{j}(x_{k})$ for any index $j \in \mathcal{I}$ and $k \in \mathcal{J}(i)$. Thus, taking the limit in $\mathcal{J}(i)$, by (\ref{teorema_principal_eq_02}) it follows that 
$
    f_{i}(x_{*}) \leq f_{j}(x_{*}),
$
and therefore $i \in \imin{(x_{*})}$.
\end{proof}

Theorem \ref{teorema_principal} states that if $\eta > 0$ and $x_{*} \in \mathbb{R}^{n}$ is an accumulation point of a sequence $\{ x_{ k} \}_{k \in \mathbb{N}}$ generated by Algorithm \ref{alg_lowder}, it is possible to construct a subsequence $\{ x_{k} \}_{k \in \mathcal{J}(i) }$ which converges to $x_{*}$, where $\mathcal{J}(i) = \{k \in \mathbb{N} \text{ }|\text{ } i_{k} = i\}$, and $x^{*}$ satisfies a necessary optimality condition of gradient projected type for the problem
\begin{equation}\label{eq_definicao_prob_fi}
    \begin{array}{cc}
        \text{minimize} & f_{i}(x) \\
        \text{subject to} & x \in \Omega.
    \end{array}
\end{equation}
In other words, $x^{*}$ is a first-order stationary point for the problem (\ref{eq_definicao_prob_fi}) (see \cite{CGT1988} and \cite[p. 450]{CGT2000}).

The following definition and corollaries help us to understand Theorem \ref{teorema_principal} from the perspective of the LOVO theory~\cite{AMMY2009}. 

\begin{defn}\cite[p. 05]{AMMY2009}\label{definicao_criticidade_fraca_forte}
Given $x_{*} \in \Omega$, we say that
\begin{enumerate}
    \renewcommand{\theenumi}{\roman{enumi}}
    \renewcommand\labelenumi{(\theenumi)}
    \item $x_{*}$ is strongly critical if it satisfies the necessary optimality condition for (\ref{eq_definicao_prob_fi}), for all $i \in \imin(x_*)$.
    \item $x_{*}$ is weakly critical if it satisfies the necessary optimality condition for
    (\ref{eq_definicao_prob_fi}), for some index $i \in \imin(x_{*})$.
\end{enumerate}
\end{defn}
 
\begin{corollary} \label{corollary_weak1}
Suppose that $\eta > 0$. If $x_{*} \in \mathbb{R}^{n}$ is an accumulation point of a sequence $\{x_{k}\}_{k \in \mathbb{N}}$ generated by Algorithm \ref{alg_lowder}, then $x_{*}$ is weakly critical. 
\end{corollary}
\begin{proof}
  Let $x_{*}$ be an accumulation point of the sequence $\{ x_{k} \}_{k \in \mathbb{N}}$ generated by Algorithm \ref{alg_lowder}. Thus, there is a subsequence $\{ x_{k} \}_{k \in \mathcal{N}} \subseteq \{ x_{k} \}_{k \in \mathbb{N}}$ that converges to $x_{*}$. If $i\in \mathcal{I}$ is an index that repeats for an infinite number of iterations in this subsequence then, as
  subsequence $\{ x_{k} \}_{k \in \mathcal{J}(i)} \subseteq \{ x_{k} \}_{k \in \mathcal{N}} \subseteq \{ x_{k} \}_{k \in \mathbb{N}}$ converges to $x_{*}$, by~\eqref{t:princ:ponto3} of Theorem~\ref{teorema_principal} it follows that $i \in \imin{(x_{*})}$. Furthermore, from~\eqref{t:princ:ponto2} of the same theorem, it follows that
\begin{equation*}
    \lim_{k \in \mathcal{J}(i)} \norm{\mathcal{P}_{\Omega}(x_{k} - \nabla f_ {i}(x_{k})) - x_{k}} = 0,
\end{equation*}
and thus $x_{*}$ satisfies a necessary optimality condition for problem (\ref{eq_definicao_prob_fi}). Therefore, $x_{*}$ is a weakly critical point.
\end{proof}

\begin{corollary} \label{corollary_weak2}
If Algorithm \ref{alg_lowder} generates a sequence $\{x_{k}\}_{k \in \mathbb{N}}$ that converges to $x_{*} \in \mathbb{R}^{n }$, then $x_{*}$ is weakly critical. 
\end{corollary}
\begin{proof}
Let $\{ x_{k} \}_{k \in \mathbb{N}}$ be a sequence generated by Algorithm \ref{alg_lowder} and suppose that it converges to a point $x_{*} \in \mathbb{ R}^{n}$. If
\begin{itemize}
    \item $\eta = 0$, let $i \in \mathcal{I}$ be the index from statement~\eqref{t:princ:ponto1} of Theorem \ref{teorema_principal};
    \item $\eta > 0$, let $i \in \mathcal{I}$ be any index chosen for an infinite number of iterations; 
\end{itemize}
Note that $\{ x_{k} \}_{k \in \mathcal{J}(i)} \subseteq \{ x_{k} \}_{k \in \mathbb{N}}$ converges to $x_{*}$ and so, by~\eqref{t:princ:ponto3} of Theorem~\ref{teorema_principal}, it follows that $i \in \imin{(x_{*})}$. Furthermore, by~\eqref{t:princ:ponto1} or~\eqref{t:princ:ponto2} of the same theorem, it follows that $\lim_{k \in \mathcal{J}(i)} \norm{\mathcal{P}_{\Omega}(x_{k } - \nabla f_{i}(x_{k})) - x_{k}} = 0$ and thus $x_{*}$ satisfies a necessary optimality condition for the problem (\ref{eq_definicao_prob_fi}). Therefore, $x_{*}$ is a weakly critical point.
\end{proof}


\begin{remark}
    Theorem~\ref{teorema_principal} highlights another very desirable feature of Algorithm~\ref{alg_lowder}. It can be used as the inner solver for general-constrained derivative-free LOVO algorithms, especially of the Inexact Restoration type~\cite{BFMS2013,FKSS2017}. In this case, the subproblems are composed of linearizations of the equality and inequality constraints at $x_k$, which means that $\Omega$ is composed of linear constraints. Given a decreasing stopping criteria $\varepsilon_k \to 0$, Algorithm~\ref{alg_lowder} solves~\eqref{eq_lovo_sem_derivadas} up to precision $\varepsilon_k$ and this sequence of solutions satisfies the so-called \emph{Approximate Gradient Projection} (AGP) condition. The AGP condition is a strong practical necessary optimality condition for constrained optimization~\cite{AHM2011}, which is satisfied by every local minimizer of a nonlinear optimization problem without any constraint qualification. Besides the theoretical results, in practice the use of the orthogonal projection $P_\Omega$ is appealing only when it is not computationally expensive. That is the case for bound constraints, for which $P_\Omega(x) = \max\{l, \min\{x, u\}\}$ (in element-wise form), the ball of radius $R$, for which $P_\Omega(x) = x$ if $\left\|x\right\| \le R$ or $P_\Omega(x) = R \cdot x / \left\|x\right\|$, otherwise, and linear constraints, for which fast iterative methods exist~\cite{BirginMartinezRaydan2003}.

\end{remark}

\begin{remark}
    Assumption~\ref{hip_grad_limitante} is instrument to obtaining global convergence results, and one may ask if such condition is achievable in practice. It is well known that Assumption~\ref{hip_grad_limitante} is satisfied if, for example, linear or quadratic interpolation models are constructed using the so-called $\Lambda$-poised sample points~\cite{CSV2009a}, $\Lambda > 0$. Models with a relaxed interpolation condition were also considered in~\cite{Schwertner2025, VKPS2017}.
    Unfortunately, usually the techniques used to ensure well poisedness only deal with the unconstrained case. In~\cite{P2009}  a technique to build good underdetermined quadratic interpolation models was used for box constraints. The strategy starts with feasible points and never allows a point to be outside the bound constraints. The same idea is used for the practical implementation of Algorithm~\ref{alg_lowder} in Section~\ref{sec:experiments}. The authors in~\cite{HR2021} present a weaker definition of ``fully linear models'' and extend the concept of $\Lambda$-poisedness. Unfortunately, this definition depends on a different stationarity measure, which requires a convex-constrained optimization problem to be solved in order to be verified. On the other hand, the projected gradient measure used by Algorithm~\ref{alg_lowder} and by Theorem~\ref{teorema_principal} has a closed form in many special cases, when $\Omega$ is a box or a hyper-sphere, for example. Another possibility to build well poised models using unconstrained strategies is to allow $\fmin$ to be evaluated in points outside $\Omega$. That was the case in~\cite{CKP2015,FKSS2017}. When $\fmin$ is not defined outside $\Omega$ (also known as \emph{hard constraints}), then the criticality measure and strategy from~\cite{HR2021} can be adapted to Algorithm~\ref{alg_lowder} and benefit from using two different radii.
\end{remark}

\section{Worst-case complexity}\label{sec:complexity}

In this section we will perform the worst-case iteration and function evaluation complexity analysis for Algorithm~\ref{alg_lowder}. For this purpose, let us first define the stationarity measure $\pi_{k}^{f} = \norm{\mathcal{P}_{\Omega}(x_{k} - \nabla f_{i_{k}}(x_{k})) - x_{k}}$, which is related to the stationarity of the convex-constrained problem~\eqref{eq_definicao_prob_fi} for \(f_{i_k}\). Given $\epsilon > 0$, we will bound the number of iterations necessary to verify \(\epsilon\)-weakly stationarity, that is $\pi_{k}^{f} < \epsilon$, in terms of the initial point, problem's and algorithm's constants, and $\epsilon$. Note that, by Theorem~\ref{teorema_principal}, this condition can be satisfied at least for a subsequence generated by Algorithm~\ref{alg_lowder}.

Let $k_{\epsilon} \in \mathbb{N}$ be the first iteration that $\pi_{k_{\epsilon}}^{f} < \epsilon$ is verified. In order to help us count the number of iterations up to $k_{\epsilon}$, we define the sets
\begin{itemize}
    \item $\mathcal{C}_{\epsilon}$, iterations $k \le k_{\epsilon}$ that entered the criticality phase (line \ref{alg:criticality});
    \item $\unsucce$, all unsuccessful iterations ($\rho_{k} < \eta$);
    \item $\accer$ and $\accenr$, all acceptable iterations ($\eta \leq \rho_{k} < \eta_{1}$) in which the radii adjustment phase (line \ref{alg:adjustment}) is and is not called, respectively;
    \item $\mathcal{\overline{S}}_{\epsilon}^{\mathcal{R}}$ and $\mathcal{\overline{S}}_{\epsilon}^{\mathcal{NR}}$, all successful iterations ($\rho_{k} \geq \eta_{1}$) in which the radii adjustment phase (line \ref{alg:adjustment}) is and is not called, respectively;
\end{itemize}

For the purposes of notation, we also define
\begin{itemize}
    \item $\succe = \succer \cup \succenr$, the set of all successful iterations;
    \item $\mathcal{R}_{\epsilon} = \unsucce \cup \accenr$, the set of all iterations in which there was necessarily a reduction in the trust-region radius;
    \item $\mathcal{N}_{\epsilon} = \mathcal{C}_{\epsilon} \cup \unsucce \cup \accer \cup \accenr = \mathcal{C}_{\epsilon} \cup \mathcal{R}_{\epsilon} \cup \accer$, the set of all iterations that are not successful (not in \(\succe\)).
\end{itemize}

The following two results help us to establish a relation between the stationarity measures $\pi_{k}$ and $\pi_{k}^{f}$.
\begin{corollary}\label{lema_medidas_estacionaridade}
  Suppose that Assumption \ref{hip_grad_limitante} holds. Then the
  stationarity measures $\pi_{k}$ and $\pi_{k}^{f}$ satisfy
  $\abs{\pi_{k} - \pi_{k}^{f}} \leq \kappa_{g}\delta_{k}$.
\end{corollary}
\begin{proof}
The result follows from the proof of Theorem \ref{teorema_principal}.
\end{proof}

\begin{lemma}\label{lema_pi_lim_inferior}
Suppose that Assumption \ref{hip_grad_limitante} holds and that Algorithm~\ref{alg_lowder} has not entered the criticality phase (line \ref{alg:criticality}) at iteration $k \in \mathbb{N}$. If $\pi_{k}^{f} \geq \epsilon$, then $\pi_{k} \geq c_{2} \epsilon$, where $c_{2} := 1 / (1 + \kappa_{g}\beta)$.
\end{lemma}
\begin{proof}
Since $k \in \mathbb{N}$ is not a criticality iteration, we have that $\delta_{k} \leq \beta \pi_{k}$. Thus, by Corollary~\ref{lema_medidas_estacionaridade}, it follows that
\begin{equation*}
    \begin{aligned}
        \epsilon \leq \pi_{k}^{f} \leq \abs{\pi_{k}^{f} - \pi_{k}} + \pi_{k} \leq \kappa_{g}\delta_{k} + \pi_{k} \leq \kappa_{g}\beta\pi_{k} + \pi_{k} \leq \left( \kappa_{g}\beta + 1 \right)\pi_{k}.
    \end{aligned}
\end{equation*}
Therefore, by letting $c_{2} = 1 / (1 + \kappa_{g}\beta)$, we obtain $\pi_{k} \geq c_{2} \epsilon$.
\end{proof}

We are now able to count the number of successful iterations. We recall Corollary~\ref{corolario_Delta_min}, which defines the \(\varepsilon\)-dependent constant $\Deltamin$, and observe that \(\varepsilon\) is a lower bound of \(\pi_k\). We apply Lemma~\ref{lema_pi_lim_inferior} with $\varepsilon = c_2 \epsilon$ in order to obtain the desired complexity results.

\begin{lemma}\label{teorema_limitante_bem_sucedidas}
  Algorithm \ref{alg_lowder} needs a maximum of
\begin{equation*}\label{teo_11_eq_00}
    \abs{\succe} 
        \leq \frac{ \fmin{(x_{0})} - M }{\theta\eta_{1}c_{2}}\Deltamin^{-1}\epsilon^{-1},
\end{equation*}
successful iterations to reach $\pi_{k_{\epsilon}}^{f} < \epsilon$, where $M := \min_{i \in \mathcal{I}}\{M_{i}\}$, \(M_i\) is defined in Assumption~\ref{hip_f_limitada} and $\Deltamin$ is defined in Corollary~\ref{corolario_Delta_min}.
\end{lemma}
\begin{proof}
Let $k \in \succe$. Using the definition of $\rho_{k}$, Lemma \ref{lema_pi_lim_inferior}, and letting $\varepsilon = c_{2}\epsilon$ in Corollary \ref{corolario_Delta_min}, we have 
\begin{equation}\label{teorema_limitante_bem_sucedidas_eq_01}
    \begin{aligned}
        \fmin{(x_{k})} - \fmin{(x_{k+1})}
        &= \rho_{k} \left(\mathrm{m}_{k}(x_{k}) - \mathrm{m}_{k}(x_{k+1})\right) \\
        &\geq \rho_{k} \theta \pi_{k} \min \left\{ \frac{\pi_{k}}{\kappa_{H}}, \Delta_{k}, 1 \right\} \\
        &\geq \eta_{1} \theta \pi_{k} \min \left\{ \frac{\pi_{k}}{\kappa_{H}}, \Delta_{k}, 1 \right\} \\
        &\geq \eta_{1} \theta c_{2} \epsilon \min \left\{ \frac{c_{2}\epsilon}{\kappa_{H}}, \Delta_{k}, 1 \right\} \\
        &\geq \eta_{1} \theta c_{2} \epsilon \min \left\{ \frac{ \tau_1 c_{2}\epsilon}{\kappa_{H}}, \Deltamin, 1 \right\} \\
        & = \theta \eta_{1} c_{2} \Deltamin \epsilon,
    \end{aligned}
\end{equation}
where the last inequality comes from the definition of $\Deltamin$ and $\tau_1 < 1$. Thus, by adding up (\ref{teorema_limitante_bem_sucedidas_eq_01}) to every $k \in \succe$, we get 
\begin{equation*}
    \begin{aligned}
        \fmin{(x_{0})} - \fmin{(x_{k_{\epsilon}})}
        &\geq \sum_{k \in \succe} \left(\fmin{(x_{k})} - \fmin{(x_{k+1})} \right) \\
        &\geq \sum_{k \in \succe} \theta \eta_{1} c_{2} \Deltamin \epsilon \\
        &= \theta \eta_{1} c_{2} \Deltamin \epsilon \abs{\succe}.
    \end{aligned}
\end{equation*}
Given that $\fmin{(x_{k_{\epsilon}})} \geq M$, then
\begin{equation*}
    \fmin{(x_{0})} - M \geq \fmin{(x_{0})} - \fmin{(x_{k_{\epsilon}})} \geq \theta \eta_{1} c_{2} \Deltamin \epsilon \abs{\succe},
\end{equation*}
and it follows that,
\begin{equation*}
        \abs{\succe} 
        \leq \frac{ \fmin{(x_{0})} - M }{\theta\eta_{1}c_{2}}\Deltamin^{-1}\epsilon^{-1}.
\end{equation*}
\end{proof}

Next, we set an upper bound on the number of iterations that are not
successful. Recall that the set of all iterations up to $k_{\epsilon}$ is given by $\succe \cup \mathcal{N}_{\epsilon}$.

\begin{lemma}\label{teorema_limitante_outras_iteracoes}
Algorithm~\ref{alg_lowder} needs at most 
\begin{equation}
    \abs{\mathcal{N}_{\epsilon}} 
    \leq \frac{\log\big( \Delta_{0} \big) - \log\big(\Deltamin\big) }{\abs{\log\big( \tau_{2} \big)}} + c_{3} \abs{\succe}
\end{equation}
not successful iterations to reach $\pi_{k_{\epsilon}}^{f} < \epsilon$, where $c_{3} = \Gamma_{\max} +  \big(\Gamma_{\max} + 1 \big) \frac{\log ( \tau_{4})}{\abs{\log\big( \tau_{2} \big)}}$.
\end{lemma}
\begin{proof}
Initially, note that, by the description of Algorithm~\ref{alg_lowder}
\begin{itemize}
    \item $\Delta_{k+1} \leq \tau_{2}\Delta_{k}$, for all $k \in \mathcal{C}_{\epsilon}$;
    \item $\Delta_{k+1} = \tau_{1}\Delta_{k}$, for all $k \in \mathcal{R}_{\epsilon}$;
    \item $\Delta_{k+1} \leq \tau_{3}\Delta_{k}$, for all $k \in \succenr$;
    \item $\Delta_{k+1} = \tau_{4}\Delta_{k}$, for all $k \in \accer \cup \succer$.
\end{itemize}
Thus, by applying Corollary~\ref{corolario_Delta_min} with
$\varepsilon = c_2 \epsilon$ and since $\tau_{1} \leq \tau_{2}$ and $\tau_{3} \leq \tau_{4}$, we must have
$
        \Deltamin \leq \Delta_{k_{\epsilon}} \leq \Delta_{0} \cdot \tau_{2}^{\abs{\mathcal{C}_{\epsilon}}} \cdot \tau_{1}^{\abs{\mathcal{R}_{\epsilon}}} \cdot \tau_{3}^{\abs{\succenr}} \cdot \tau_{4}^{\abs{\accer} + \abs{\succer}}
        \leq \Delta_{0} \cdot \tau_{2}^{\abs{\mathcal{C}_{\epsilon}} + \abs{\mathcal{R}_{\epsilon}}} \cdot \tau_{4}^{\abs{\accer} + \abs{\succer} + \abs{\succenr}}
        = \Delta_{0} \cdot \tau_{2}^{\abs{\mathcal{C}_{\epsilon}} + \abs{\mathcal{R}_{\epsilon}}} \cdot \tau_{4}^{\abs{\accer} + \abs{\succe}}
$. Therefore,
\begin{equation*}
    \begin{aligned}
        \log\big(\Deltamin\big) 
        &\leq \log{\left( \Delta_{0} \cdot \tau_{2}^{\abs{\mathcal{C}_{\epsilon}} + \abs{\mathcal{R}_{\epsilon}}} \cdot \tau_{4}^{\abs{\accer} + \abs{\succe}} \right)} \\
        &= \big( \abs{\mathcal{C}_{\epsilon}} + \abs{\mathcal{R}_{\epsilon}} \big) \log{\big( \tau_{2} \big)} + \left( \abs{\accer} + \abs{\succe} \right)\log{\big( \tau_{4} \big)} + \log{\big( \Delta_{0} \big)},
    \end{aligned}
\end{equation*}
what provides us
\begin{equation*}
    \big( \abs{\mathcal{C}_{\epsilon}} + \abs{\mathcal{R}_{\epsilon}} \big) \log{\big( \tau_{2} \big)} \geq \log{\big(\Deltamin\big)}  - \log{\big( \Delta_{0} \big)} - \left( \abs{\accer} + \abs{\succe} \right)\log{\big( \tau_{4} \big)}.
\end{equation*}
Given that $\tau_{2} < 1$, we have $\log{\big( \tau_{2} \big)} < 0$, it follows that
\begin{equation*}
    \begin{aligned}
        \abs{\mathcal{C}_{\epsilon}} + \abs{\mathcal{R}_{\epsilon}}
        & \leq \frac{\log\big(\Deltamin\big) - \log\big( \Delta_{0} \big) - \left( \abs{\accer} + \abs{\succe} \right)\log{\big( \tau_{4} \big)}}{\log{\big( \tau_{2} \big)}} \\
        & = \frac{\log{\big( \Delta_{0} \big)} - \log{\big(\Deltamin\big)} }{\abs{\log{\big( \tau_{2} \big)}}} + \left( \abs{\accer} + \abs{\succe} \right)\frac{\log{\big( \tau_{4} \big)}}{\abs{\log\big( \tau_{2} \big)}}.
    \end{aligned}
\end{equation*}
By the radii adjustment phase (line \ref{alg:adjustment}), we know that after an iteration in $\succe$, at most $\Gamma_{\max}$ iterations of the type $\accer$ can occur, and so, $\abs{\accer} \leq \Gamma_{\max} \abs{\succe}$.
Therefore, it follows that
\begin{equation*}
    \begin{aligned}
        \abs{\mathcal{N}_{\epsilon}} 
        &= \abs{\mathcal{C}_{\epsilon}}  + \abs{\mathcal{R}_{\epsilon}} + \abs{\accer} \\
        &\leq \frac{\log{\big( \Delta_{0} \big)} - \log{\big(\Deltamin\big)} }{\abs{\log{\big( \tau_{2} \big)}}} + \big( \Gamma_{\max} + 1 \big)\abs{\succe} \frac{\log{\big( \tau_{4} \big)}}{\abs{\log\big( \tau_{2} \big)}} + \Gamma_{\max}\abs{\succe} \\
        &= \left( \Gamma_{\max} +  \big(\Gamma_{\max} + 1 \big) \frac{\log{\big( \tau_{4} \big)}}{\abs{\log{\big( \tau_{2} \big)}}}\right)\abs{\succe} + \frac{\log{\big( \Delta_{0} \big)} - \log{\big(\Deltamin\big)} }{\abs{\log{\big( \tau_{2} \big)}}},
    \end{aligned}
  \end{equation*}
what, together with the definition of $c_{3}$, finishes the proof.
\end{proof}

\begin{theorem} \label{teorema_pior_caso}
Algorithm \ref{alg_lowder} needs at most $\mathcal{O}(\max\left\{ \kappa_{g}, L \right\}\kappa_{g}^{2} \epsilon^{-2})$ iterations to reach $\pi_{k_{\epsilon}}^{f} < \epsilon$.
\end{theorem}
\begin{proof}
By Lemmas~\ref{teorema_limitante_bem_sucedidas} and~\ref{teorema_limitante_outras_iteracoes}, we have
\begin{equation}\label{teorema_pior_caso_eq_01}
    \begin{aligned}
        \abs{\succe} + \abs{\mathcal{N}_{\epsilon}}
        &\leq \abs{\succe} + c_{3}\abs{\succe} + \frac{\log{\big( \Delta_{0} \big)} - \log{\big(\Deltamin\big)} }{\abs{\log{\big( \tau_{2} \big)}}} \\
        &= (1 + c_{3}) \abs{\succe} + \frac{\log\big( \Delta_{0} \big) - \log\big(\Deltamin\big) }{\abs{\log\big( \tau_{2} \big)}} \\
        & \leq \frac{\left( \fmin{(x_{0})} - M \right)}{\theta\eta_{1} c_{2}} c_{3} \Deltamin^{-1}\epsilon^{-1} + \frac{\log\big( \Delta_{0} \big) + \log\big(\Deltamin^{-1}\big) }{\abs{\log\big( \tau_{2} \big)}}.
\end{aligned}
\end{equation}
From Lemma \ref{lema_Delta_sucesso}, we have that $c_{1} = \left(L + \kappa_{g} + \kappa_{H} / 2 \right) \theta^{-1} = \mathcal{O}(\max\left\{ \kappa_{g}, L \right\})$, since $\kappa_{H} = 2\kappa_{g} + L + 1 = \mathcal{O}(\max\{\kappa_g, L\})$. Similarly, by Lemma \ref{lema_pi_lim_inferior}, $c_{2}^{-1} = 1 + \kappa_{g}\beta = \mathcal{O}(\kappa_{g})$. Hence, by letting $\varepsilon = c_{2}\epsilon$ in Corollary \ref{corolario_Delta_min},
\begin{equation}\label{teorema_pior_caso_eq_03}
    \begin{aligned}
        \Deltamin^{-1} 
        & = \max\left\{ \Delta_{0}^{-1}, \frac{\kappa_{H}}{\tau_{1} c_{2} \epsilon}, \frac{c_{1}}{\tau_{1} (1 - \eta_{1}) c_{2} \epsilon}, \frac{1}{\tau_{1} \beta c_{2} \epsilon}, \tau_{1}^{-1} \right\} \\
        & = \mathcal{O}(\max\left\{ \kappa_{g}, L \right\}\kappa_{g}\epsilon^{-1}).
    \end{aligned}
\end{equation}
Therefore, we get from~\eqref{teorema_pior_caso_eq_01} and~\eqref{teorema_pior_caso_eq_03} that
\(
\abs{\succe} + \abs{\mathcal{N}_{\epsilon}} = \mathcal{O}(\max\left\{ \kappa_{g}, L \right\}\kappa_{g}^{2}\epsilon^{-2})
\).
\end{proof}

\begin{remark}
Similarly to \cite[Algorithm 3.2]{AMM2008}, we can allow resets in the trust-region radius of type $\Delta_{k + 1} = \max \left(\tau_{3} \Delta_{k}, \Delta_{0}\right)$ at line \ref{alg:adjustment_increase} in the radii adjustment phase. However, this choice worsens the upper bound presented in Theorem \ref{teorema_pior_caso} to $\mathcal{O}(\epsilon^{-3})$.
\end{remark}

\begin{remark}
Similarly to~\cite{GJV2016}, we chose not to include assumptions about the order of magnitude of the $\beta$ parameter, which acts in accessing the criticality phase (line \ref{alg:criticality}).
Although it is desirable that $\beta$ is taken in an inversely proportional way to the choice of $\epsilon$, in the context of our algorithm that is not necessary to establish complexity results. In~\cite{CR2019}, deriving worst-case complexity results for algorithm {\tt DFO-GN} requires an assumption about the magnitude of the criticality phase parameter.
\end{remark}

In the following result, we explicitly expand constant $\kappa_{g}$ to show the worst-case complexity estimates for function evaluations. This is possible by further specifying how models $\mathrm{m}$ are constructed. We consider determined and underdetermined linear and quadratic polynomial models satisfying the inexact interpolation condition for each point $y^{j}$ in a $\Lambda$-poised sample set $\mathcal{Y} \subset \overline{B}(x_{k}, \delta_k)$, where $\kappa \ge 0$ is the inexact constant. Such calculations were given in~\cite{Schwertner2025}. Note that this interpolation condition naturally includes the classical interpolation models when $\kappa = 0$. Comments on how to build and maintain $\Lambda$-poised sets were made at the remarks in the end of Section~\ref{sec:convergence} and practical considerations are subject to the next section.

\begin{theorem}\label{teorema_pior_caso_aval_funcoes}
Assume that the models $\mathrm{m}_k$ are constructed by inexact linear or quadratic interpolation using a $\Lambda$-poised set $\mathcal{Y} \subset \mathbb{R}^n$ of sample points, $\Lambda > 0$. Then, the number of function evaluations that Algorithm~\ref{alg_lowder} needs in order to reach $\pi_{k_{\epsilon}}^{f} \leq \epsilon$ is
\begin{enumerate}
    \renewcommand{\theenumi}{\roman{enumi}}
    \renewcommand\labelenumi{(\theenumi)}
    \item $\mathcal{O}\big( (r + n) L^3 n^{3} \epsilon^{-2} \big)$ if the model is linear;
    \item $\mathcal{O} \big( (r + n^{2}) L^3 n^{12} \epsilon^ {-2} \big)$ if the model is quadratic determined;
    \item $\mathcal{O} \big( (r + p) L^3 n^{\frac{9}{2}} p^{\frac{15}{2}} \epsilon^{-2 } \big)$ if the model is quadratic underdetermined, for $n < p < (n^{2} + 3n) / 2$ is the number of points used.
\end{enumerate}
\end{theorem}
\begin{proof}
Initially, note that at each iteration, at most $\abs{\mathcal{Y}}$ evaluations of the function $f_{i}$ are performed, for some index $i \in \mathcal{I}$, in order to build a $\Lambda$-poised set from scratch. Moreover, a single evaluation of the $\fmin$ function is necessary to evaluate $\rho_{k}$, which in turn depends on $\abs{\mathcal{I}} = r$ function evaluations. Thus, by Theorem \ref{teorema_pior_caso}, we will need at most $\mathcal{O} \big( ( \abs{\mathcal{Y}} + r ) \max\{\kappa_g, L\} \kappa_{g}^{2} \epsilon^{-2} \big)$ function evaluations during algorithm execution. We will separate the proof into three cases.

\begin{enumerate}
\renewcommand{\theenumi}{\roman{enumi}}
\renewcommand\labelenumi{(\theenumi)}
\item Linear case.

In this case $\abs{\mathcal{Y}} = n + 1$ and, by \cite[Theorem 1]{Schwertner2025} and \cite[Lemma 2]{Schwertner2025}, we have that
\begin{equation*}
    \kappa_{g} = L + \left( \frac{L}{2} + 2 \kappa \right)\Lambda n = \mathcal{O}(Ln),
\end{equation*}
what gives us $\max\{\kappa_g, L\} = \mathcal{O}(Ln)$. Thus, we conclude that at most $\mathcal{O}\big( (n + r) L^3n^{3} \epsilon^{-2} \big)$ function evaluations are necessary.

\item \label{teorema_pior_caso_aval_funcoes:casoii} Quadratic case.

  In this case $\abs{\mathcal{Y}} = (n^{2} + 3n) / 2 + 1 = q + 1$. Again, using \cite[Theorem 1]{Schwertner2025} and \cite[Lemma 2]{Schwertner2025}, we obtain
\begin{equation*} 
    \begin{aligned}
        \kappa_{g}
        &= 2 \Big( 4 \Lambda \sqrt{(q+1)^{3}} \Big) \sqrt{q} \big(1 + \sqrt{2} \big)(\kappa + L) \\
        &= 8\Lambda \big( 1 + \sqrt{2} \big) \sqrt{q} \sqrt{(q+1)^{3}} (\kappa + L) \\
        &< 8\Lambda \big( 1 + \sqrt{2} \big) (q+1)^{2} (\kappa + L) \\
        &= 8\Lambda \big( 1 + \sqrt{2} \big) \bigg( \frac{n^{2} + 3n + 2}{2} \bigg)^{2} (\kappa + L) \\
        & = \mathcal{O}(Ln^{4})
    \end{aligned}
\end{equation*}
what gives to us at most $\mathcal{O} \big( (r + n^{2}) L^3n^{12} \epsilon^ {-2} \big)$ function evaluations.

\item Quadratic underdetermined case.

Finally, suppose that the models are quadratic underdetermined, with $\abs{\mathcal{Y}} = p + 1$ points, where $n < p < q$, $q$ defined in case~\eqref{teorema_pior_caso_aval_funcoes:casoii}. Thus, by \cite[Theorems 2, 3, 4]{Schwertner2025}, and by letting $c_{4} = \frac{\Lambda}{c(\delta_{\max})^{2}}$, it follows that 
\begin{equation*}
    \begin{aligned}
        \kappa_{g} 
        &= 2\sqrt{p}\left( \Lambda
        \sqrt{2(n + 1)} (p + 1) \right)\left(L + \kappa + c_{4} (p + 1) \sqrt{2(q+1)}\right) \\
        &< 2\sqrt{2}\Lambda
        \sqrt{(n + 1)} \sqrt{(p + 1)^{3}}\left(L + \kappa + c_{4}\sqrt{2} (p + 1) \sqrt{(q+1)}\right) \\
        &= \mathcal{O}(L n^{\frac{3}{2}}p^{\frac{5}{2}}).
    \end{aligned}
\end{equation*}
Hence, we will need at most $\mathcal{O} \big( (r + p) L^3 n^{\frac{9}{2}} p^{\frac{15}{2}} \epsilon^{-2 } \big)$ function evaluations.
\end{enumerate}
\end{proof}

It is also worth mentioning that the constant $L$ that appears in Theorem~\ref{teorema_pior_caso_aval_funcoes} may also depend on $n$. Among the works that study derivative-free trust-region methods with convergence to first-order stationary points, as Algorithm \ref{alg_lowder}, we can highlight the bound $\mathcal{O} \big( \epsilon^{-2 } \big)$ obtained by~\cite{GJV2016} for unconstrained composite optimization problems, which is also obtained in expectation by~\cite{GRVZ2018}, but based on probabilistic models. In~\cite{GYY2016}, on the other hand, the bound $\mathcal{O} \big( \abs{\log{(\epsilon)}}\epsilon^{-2 } \big)$ is presented for unconstrained composite nonsmooth problems and problems with equality constraints. Hough and Roberts~\cite{HR2021} followed a strategy very similar to the one described in this work, for general closed convex sets. The high powers of $n$ that appear on Theorem~\ref{teorema_pior_caso_aval_funcoes} highlight the known fact that derivative-free algorithms do not scale well with dimension. On the other hand, as detailed in Section~\ref{sec:experiments}, practical implementations usually compute only one new candidate at each iteration, and try to build models using already computed information. The worst case occurs when models have to be rebuilt from scratch at every unsuccessful iteration.

Recently, Álvarez, Birgin and Martínez~\cite{Alvarez2025} provided, for the first time,  worst case complexity results for LOVO algorithms, in the usual unconstrained and convex-constrained case with derivatives. The same worst-case iteration complexity of $O(\epsilon^{-2})$ was  obtained, where the constants also depended on the sufficient decrease constant, $\fmin(x_0)$ and an approximation of the Lipschitz constant of the functions. Regarding the worst-case \emph{evaluation} complexity, although $O(\epsilon^{-2})$ was also obtained, the dependence on the dimension of the problem, which is usual in the derivative-free context, was not explicit. In our setting, the domain must be sufficiently explored in order to build models satisfying Assumption~\ref{hip_grad_limitante}, hence the dimension $n$ explicitly appears. The bound also depends on the number of functions ($r$), since we build models for $f_i$, instead of models for $\fmin$.

\section{Numerical implementation and experiments}\label{sec:experiments}

In this section, we will discuss some implementation details of Algorithm~\ref{alg_lowder}, which we will call \texttt{LOWDER}, an acronym for Low order-value Optimization With\-out DERivatives. We will also present the test problem sets adopted and the performance of our algorithm. 


\subsection{Implementation details}

The \texttt{LOWDER} algorithm employs linear models to solve LOVO problems and has several practical improvements when compared to the theoretical algorithm presented in Section \ref{sec:df_algorithm}. One of the most significant changes occurs in definition of the relative reduction coefficient $\rho_{k}$. Note that we need to calculate $\rho_{k}$ at each iteration in order to define the step acceptance, update, and radius correction phases. Since this involves evaluating the function $\fmin$, which can be costly in computational terms, in practice, we allow \texttt{LOWDER} to employ the coefficient proposed by~\cite{CLSS2021}:
\begin{equation*}
    \hat{\rho}_{k} = \frac{f_{i_{k}}(x_{k}) - f_{i_{k}}(x_{k}+d_{k})}{\mathrm{m}_{k}(x_{k}) - \mathrm{m}_{k}(x_{k} + d_{k})},
\end{equation*}
where $i_{k} \in \imin{(x_{k})}$, for up to $\texttt{nrhomax} \in \mathbb{N}$ consecutive iterations. After $\texttt{nrhomax}$ iterations $\rho_{k}$ has to be calculated as presented in (\ref{eq_def_coef_reducao_relativa}). Note that $\hat{\rho}_{k} \leq \rho_{k}$, and so we are being more demanding about the quality of the models and the decrease obtained. During the preliminary tests, $\texttt{nrhomax} = 3$ proved to be efficient.

Another relevant change is the way the solution $d_{k} \in \mathbb{R}^{n}$ of the trust region subproblem (\ref{eq_def_subproblema}) is incorporated into the sample set $\mathcal{Y }$. To do so, we implemented simplified versions of the \texttt{TRSBOX} and \texttt{ALTMOV} functions, developed by~\cite{P2009} for the \texttt{BOBYQA} algorithm. Initially, we compute a solution $d_{k}^{\texttt{TRS}}$ using the \texttt{TRSBOX} algorithm for the linear case. If $\norm{d_{k}^{\texttt{TRS}}} \geq \frac{\Delta_{k}}{2}$ and $\rho_{k} > 0$ (or $\hat{\rho}_{ k} > 0$), the point $x_{new} = x_{k} + d_{k}^{\texttt{TRS}}$ is inserted into $\mathcal{Y}$. Otherwise, we run \texttt{ALTMOV} to look for an alternative direction $d_{k}^{\texttt{ALT}}$, and add the point $x_{new} = x_{k} + d_{k}^{\texttt{ALT}}$ to $\mathcal{Y}$. At each iteration of \texttt{LOWDER}, only one point is added to $\mathcal{Y}$, while another point is removed, following procedures similar to \texttt{BOBYQA}.

In this first version, \texttt{LOWDER} employs only linear models to solve problem (\ref{eq_lovo_sem_derivadas}). In this sense, \texttt{LOWDER} uses the sample set construction mechanism of \texttt{BOBYQA}. As with this solver, we avoided fully rebuilding the models due to the computational cost involved. In general, the model and sample set are only rebuilt from scratch if one of the following conditions is satisfied: 
\begin{enumerate}
\renewcommand{\theenumi}{\roman{enumi}}
\renewcommand\labelenumi{(\theenumi)}
\item An index swap occurs, that is $i_{k+1}\neq i_{k}$, and the direction calculated by \texttt{TRSBOX} satisfies $\rho_{k} \geq \eta$ (or $\hat{\rho}_{k} \geq \eta$);  
\item An index swap occurs, $\rho_{k} \geq 0$ (or $\hat{\rho}_{k} \geq 0$) and $\fmin{(x_{new})}$ is available. 
\end{enumerate} 
In all other cases, the model is updated using the current sample set and the information obtained about the point $x_{new}$. Recall that only evaluations of $f_{i_{k+1}}$ are needed to rebuild the model. Each time the model is rebuilt or updated, we also calculate the QR factorization of the matrix \[\mathbf{L}_{L} = \begin{bmatrix}
    (y^1 - y^0)^T \\ \vdots \\ (y^n - y^0)^T
\end{bmatrix}\] defined in \cite[Section 2]{Schwertner2025}. This information is needed to calculate $d^{\texttt{ALT}}_{k}$ given by \texttt{ALTMOV} and the point that should leave set $\mathcal{Y}$ in case direction $d^{\texttt{TRS}}_{k}$ is accepted.

Given that the radius of the sample region $\delta_{k}$ controls the quality of the model, a point $x_{k} \in \Omega$ is declared solution if it satisfies $\delta_{k} \leq \delta_{\min}$ and $\beta \pi_{k} \leq \delta_{\min}$ (or $\beta \Tilde{\pi}_{k} \leq \delta_{\min}$), where $\delta_{\min} > 0$ is a parameter defined by the user. The default values for $\beta$ and $\delta_{\min}$ are $1$ and $10^{-8}$, respectively. We say that \texttt{LOWDER} is stalled at iteration $k$ if there is no more room for improvement in the sample set and the model. This situation translates into the case where $\delta_{k} \leq \delta_{\min}$, $\Delta_{k} \leq \delta_{\min}$ and it is not possible to perform any more \texttt{ALTMOV} iterations. In this case, the maximum number of consecutive iterations of the \texttt{ALTMOV} type is controlled by the $\texttt{maxalt} \in \mathbb{N}$ parameter, whose default value is $\abs{\mathcal{Y}}-1$. Note that if the sample and trust-region radii have values greater than $\delta_{\min}$, we allow more than \texttt{maxalt} consecutive iterations with \texttt{ALTMOV} directions. However, as soon as the stalled criterion is satisfied the algorithm exits. For safety, the algorithm also exits if more than $\texttt{maxcrit} \in \mathbb{N}$ successive iterations access the criticality phase (line \ref{alg:criticality}), whose default value is $\abs{\mathcal{Y}}-1$.

\texttt{LOWDER} was implemented in the \texttt{Julia} language~\cite{Bezanson2017}, version 1.6.1, and is available in the repository:
\begin{center}
    \url{https://github.com/aschwertner/LOWDER}
\end{center}

\subsection{Numerical experiments}

To benchmark the \texttt{LOWDER}'s performance in solving LOVO black-box problems, we compared our solver with other algorithms able of solving derivative-free optimization problems. These algorithms are Manifold Sampling Primal (\texttt{MS-P}) and Nonlinear Optimization with the MADS algorithm (\texttt{NOMAD}).

The \texttt{MS-P} solver is a \texttt{MatLab} implementation of the Manifold Sampling method proposed in \cite{LMW2019, LM2022}, and is designed to solve bound-constrained nonsmooth composite minimization problems
\begin{equation*}
    \min_{x \in \Omega} f(x) = \min_{x \in \Omega} h\left(F(x)\right),
\end{equation*}
where $F:\mathbb{R}^{n} \to \mathbb{R}^{p}$, $h:\mathbb{R}^{p} \to \mathbb{R}$ is a continuous selection (see \cite[Definition 1]{LM2022}), and the feasible set is a subset of the $n$-fold Cartesian product of the extended reals $\mathbb{R} \cup \{-\infty, \infty \}$ defined by bound constraints of the form $\Omega = \left\{x : l \leq x \leq u  \right\}$. The manifold sampling algorithm was initially proposed by~\cite{LMW2019} as a variant of gradient sampling, and can be classified as model-based derivative-free method. In this method, models of $F$ are combined with sampled information about the function $h$ to construct local models called smooth master models, for use within a trust-region framework. \texttt{MS-P} builds fully linear quadratic models using the interpolation and regression mechanisms of \texttt{POUNDERS}~\cite{W2017}.

The \texttt{NOMAD} software is a \texttt{C++} implementation of the Mesh Adaptive Direct Search (MADS) algorithm, solves black-box optimization problems in general, and uses the progressive barrier method to deal with problems with constraints. In this work, we use version 4.2.0 of \texttt{NOMAD} \cite{ADMT2022} through the interface for the \texttt{Julia} language called \texttt{NOMAD.jl}, version 2.2.1 \cite{MPS2020}. This version of \texttt{NOMAD} presents a series of improvements compared to previous versions, especially in comparison to version 3.9, such as a new software architecture and support for new pool mechanisms and search algorithms, among others. All the algorithms were run with their default parameters on an AMD Ryzen 7 1700X 3.40GHz with 8 cores (16 threads) and 16GB of RAM and Linux Ubuntu Budgie 22.04.1 LTS operating system.

Our benchmark suite includes three sets of test problems denominated by MW, HS, and QD. The MW set contains problems proposed by Moré and Wild in \cite{MW2009} for benchmarking derivative-free optimization algorithms and comprises $53$ unconstrained problems. These problems are variations of $22$ nonlinear least squares functions taken from the \texttt{CUTEr} collection \cite{GOT2003}. Each function is defined by $r$ components in $n$ variables. Instead of solving a usual smooth least squares problem, we define the objective function using a LOVO approach, that is $f(x) = \min \limits_{k = 1, \dots, r} f_k(x)^2$. By combining different values for $n$ and $r$, and distinct starting points, we obtain all the problems in MW test set.

The HS test set includes $87$ bound-constrained problems selected and modified from the original collection published by~\cite{HS1980} for testing nonlinear programming algorithms. Initially, we selected $8$ problems from \cite{HS1980} with bound constraints and different objective functions.
By combining two, three, or four of these problems, a new objective function $\fmin$ is defined. The problem dimension is the largest dimension among the combined problems, and the bound constraints are taken as the intersection of the original boxes as long as they do not create infeasibility or fixed variables.

Lastly, the QD test set has $5$ subsets with $50$ problems each, all with bound constraints. The purpose of this test set is to demonstrate the impact of the number of functions $f_{i}$ composing the objective function $\fmin$ in the performance of the algorithms. In this sense, the subsets of QD consist of problems with an increasing number of component functions based on the general formula:
\begin{equation*}
    f_{i}(x) = 5^{i} + \frac{1}{2} \sum_{j=1}^{n} a^{i}_{j} \cdot ( x_{j} - b^{i}_{j} )^2,
\end{equation*}
where $a^{i} \in [0, 1000]^{n}$, and $b^{i} \in [0, 10]^{n}$ are vectors built by the pseudorandom number generator \texttt{MarsenneTwister} from the \texttt{Random} package of the \texttt{Julia} language, with the same seed. 
We generate the problems in the test subsets with $r \in \{10, 25, 50, 75, 100 \}$ component functions, respectively. We employ the same seed to generate vectors $a^{i}$ and $b^{i}$ to ensure that the generated component functions come into problems belonging to test sets with larger $r$ values. In this way, the component functions generated for the QD10 test subset are present in the other subsets, and they are complemented by new functions $f_{i}$ as the value of $r$ increases.
In our case, all problems were generated with dimension $n = 10$, have the same bound constraints $l = [0, \dots, 0]$ and $u = [10, \dots, 10]$, and the starting point $x_{0} \in \mathbb{R}^{n}$ is set to the center of the box. Figure \ref{imagem_exemplo_QD} ilustrates a two dimensional example of objective function generated by this procedure, with $10$ component functions. 

\begin{figure}[ht]
    \centering
    \begin{minipage}{0.5\textwidth}
        \centering
        \includegraphics[width=0.95\textwidth]{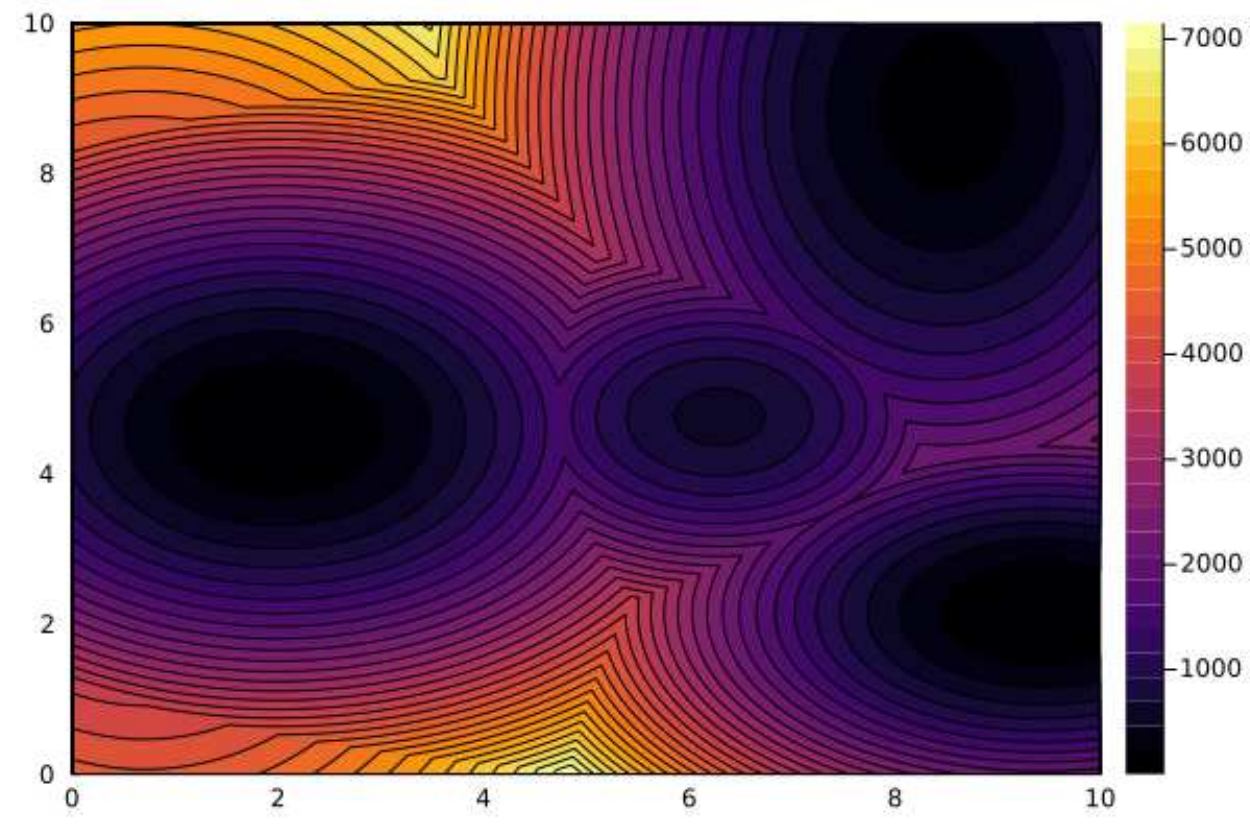}
    \end{minipage}\hfill
    \begin{minipage}{0.5\textwidth}
        \centering
        \includegraphics[width=0.95\textwidth]{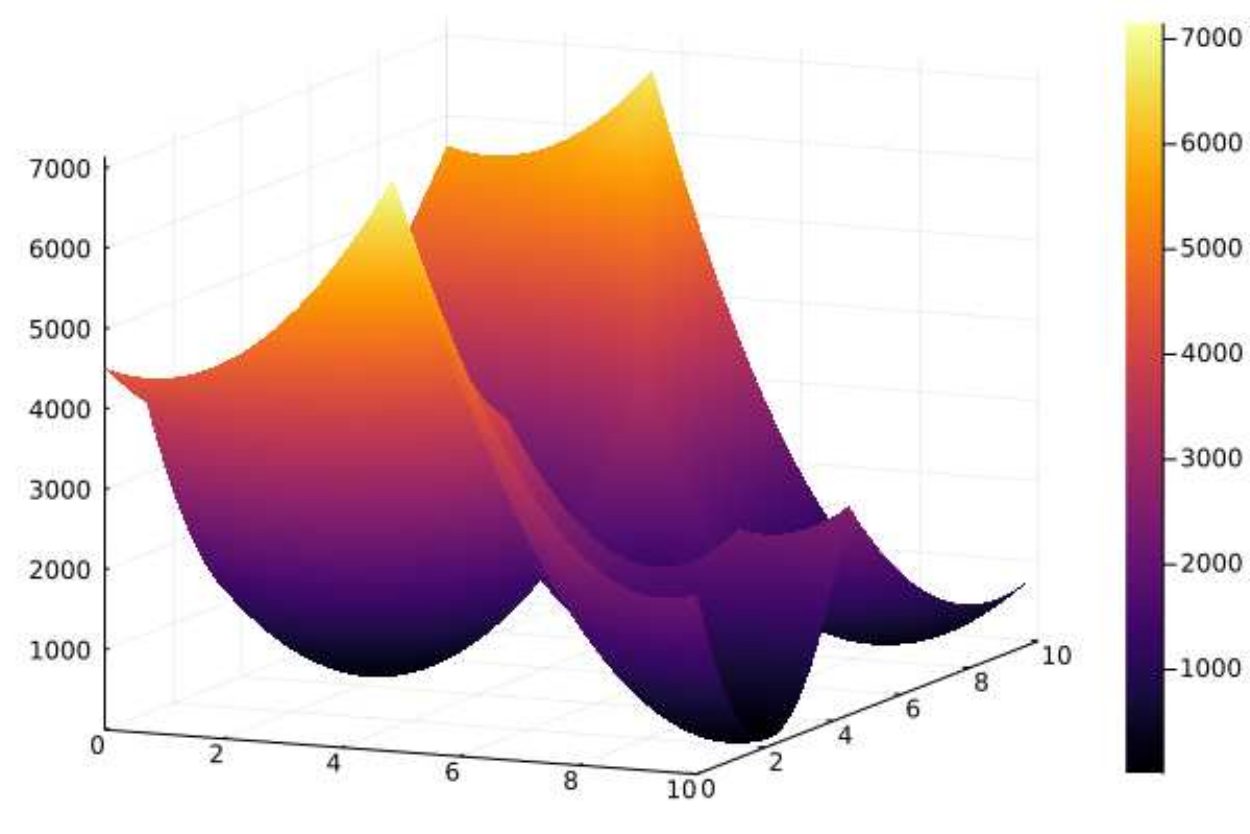}
    \end{minipage}
    \caption[Contour and surface plot of an example of QD problem.]{Contour plot (left) and surface plot (right) for an example of objective function $\fmin$ defined by the generating function of QD test set.}
    \label{imagem_exemplo_QD}
\end{figure}

We also present in Table \ref{tabela_informacoes_conjuntos_teste} the maximum and minimum values for the dimension $n \in \mathbb{N}$ and the number of component functions $r \in \mathbb{N}$ of the problems of each test set. Detailed information about the considered problems can be obtained in~\cite{Schwertner2023}.

\begin{table}[ht]
    \centering
    \begin{tabular}{@{}ccccc@{}}
    \toprule
    Test set & $n_{\min}$ & $n_{\max}$ & $r_{\min}$ & $r_{\max}$ \\ \midrule
    MW       & $2$  & $12$ & $2$  & $65$  \\
    HS       & $2$  & $10$ & $2$  & $4$ \\
    QD       & $10$ & $10$ & $10$ & $100$ \\ \bottomrule
    \end{tabular}
    \caption[Dimension and number of component functions of the test sets.]{Dimension and number of component functions of MW, HS and QD test sets.}
    \label{tabela_informacoes_conjuntos_teste}
\end{table}

We use data profiles \cite{MW2009} to compare the performance of \texttt{LOWDER}, \texttt{MS-P}, and \texttt{NOMAD}, over our benchmark test suite. We are interested in comparing the function values obtained by each of the algorithms. Thus, we consider that a method has solved a problem with tolerance level $\tau > 0$ after $t$ function evaluations if the iterate $x^{t}$ satisfies 
\begin{equation}\label{eq_criterio_data_profile}
    f(x^{t}) \leq f_{L} + \tau \left( f(x_{0}) - f_{L} \right),
\end{equation}
where $x_{0}$ is the starting point of the problem common to all algorithms, and $f_{L}$ is the smallest function value obtained by the solvers for a given budget of function evaluations. Following Moré and Wild's suggestion \cite{MW2009}, we decided to investigate the behaviour of algorithms with a limit of $100$ simplex gradients of the objective function, where one simplex gradient is defined by $n+1$ function evaluations. Since \texttt{MS-P} and \texttt{NOMAD} always evaluate function $\fmin$ completely and \texttt{LOWDER} has the ability to run each component function $f_{i}$, $i \in \mathcal{I}$, independently, for each set of tests considered, we define the budget as $100\left( n_{max} + 1 \right)$ evaluations of $\fmin$ for \texttt{MS-P} and \texttt{NOMAD}, and $100r_{p}\left( n_ {max} + 1 \right)$ evaluations of $f_{i}$ for \texttt{LOWDER}, where $r_{p}$ is the number of component functions of the problem $p$. Therefore, we allow the algorithms to run at least $100$ simplex gradients of the objective function $\fmin$ for each problem in the test set.

To build the data profiles, we recorded the function values of $\fmin$ accessed by the \texttt{MS-P} and \texttt{NOMAD} solvers and the function values of $f_{i}$ computed by \texttt{LOWDER} for each one of the selected problems. Once the solver satisfies the criterion (\ref{eq_criterio_data_profile}) for a problem $p$ for the first time after $t$ function evaluations, its data profile row is incremented on the vertical axis by $\frac{1}{\abs{\mathcal{P}}}$ at the point $\frac{t}{(n_{p} + 1)}$, in the case of \texttt{MS-P} and \texttt{NOMAD}, or at the point $\frac{t}{r_{p}(n_{p} + 1)}$, for \texttt{LOWDER}. Therefore, our metric of simplex gradient evaluations for $\fmin$ is maintained.

The codes used to generate the numerical tests and data profiles presented in this chapter are available at:
\begin{center}
    \url{https://github.com/aschwertner/LOWDER_Numerical_Tests}
\end{center}

\subsubsection{MW test set}

The MW set is our main problem test set since it was designed by~\cite{MW2009}, especially to benchmark unconstrained derivative-free algorithms. In Figure \ref{imagem_data_profile_mw}, we present four data profiles for the MW test set. Each plot shows the percentage of problems solved for a specified tolerance $\tau$ as a function of a computational budget of simplex gradients of $\fmin$. As we can see, \texttt{LOWDER} and \texttt{MS-P} solve around $90\%$ of the problems for all tolerance levels $\tau$, up to the fixed budget. The difference in robustness between these two solvers is less than $5\%$. We also note that the behaviour of \texttt{NOMAD} is very different for distinct levels of tolerance, ranging from around $60\%$ to less than $40\%$, with $\tau = 10^{-1}$ and $\tau = 10^{-7}$, respectively.

\texttt{LOWDER} is very efficient at solving problems in the MW test suite, solving about $90\%$ of the problems with less than $20$ simplex gradients. \texttt{MS-P} has a behaviour very close to \texttt{LOWDER}, outperforming the latter by about $2\%$ for $\tau = 10^{-1}$, but loses performance as the tolerance decreases, solving only $80\%$ of problems for $\tau = 10^{-7}$. Although this behaviour is expected, since \texttt{LOWDER} uses mainly information of the $f_{i}$'s, we must recall that \texttt{MS-P} uses quadratic models, while \texttt{LOWDER} uses linear ones. For this computational budget and $\tau = 10^{-1}$, \texttt{NOMAD} can solve $50\%$ of the problems in the test set. However, its performance drops and stabilizes at around $40\%$ for other values of tolerance. Inspecting the \texttt{NOMAD} execution data, we can see three distinct behaviours that can explain its low performance in the MW test set: difficulty in reducing the objective function, low convergence rate, and convergence to weak stationary points that are local minima.

\begin{figure}[ht]
    \centering
    \begin{minipage}{0.5\textwidth}
        \centering
        \includegraphics[width=0.95\textwidth]{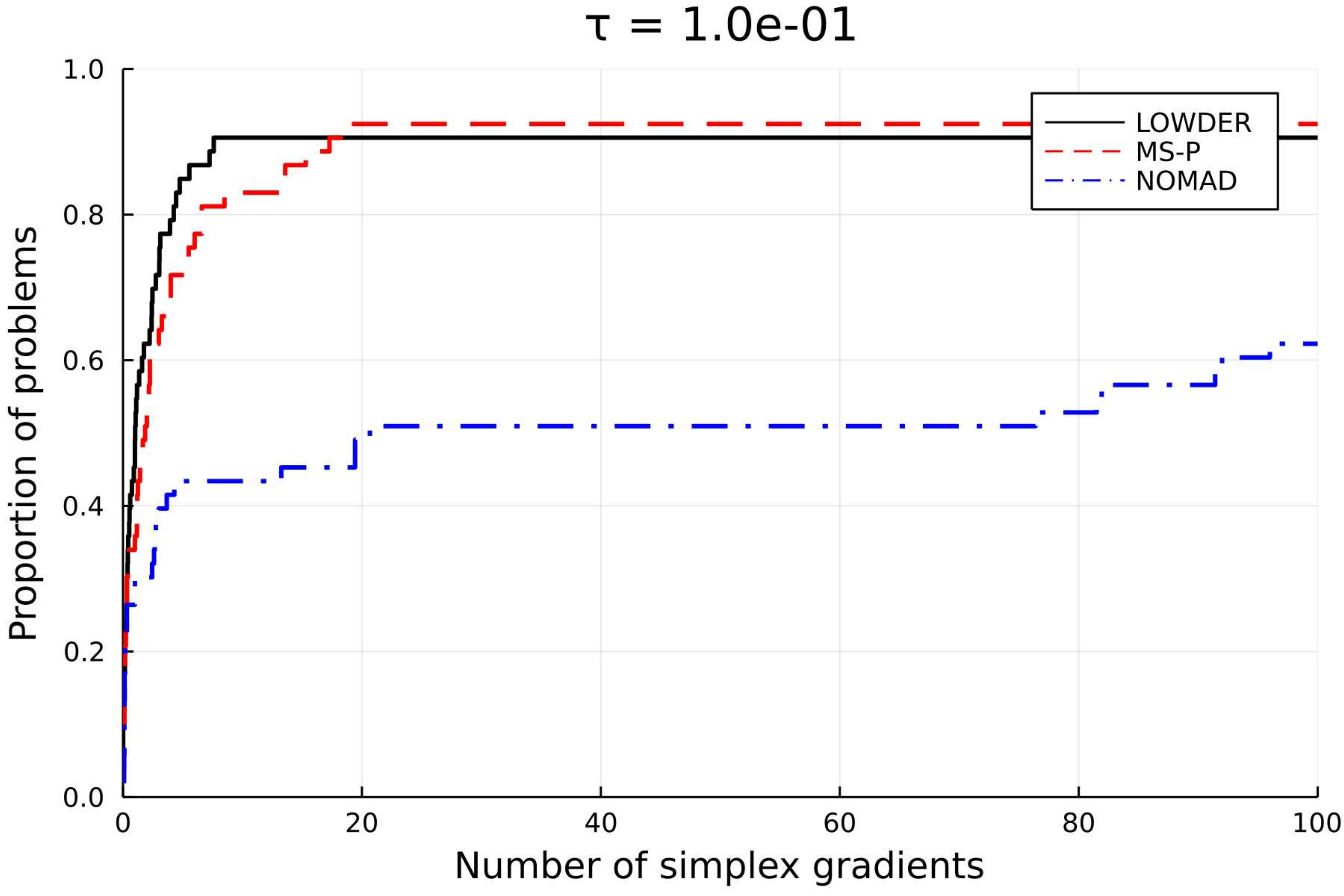}
    \end{minipage}\hfill
    \begin{minipage}{0.5\textwidth}
        \centering
        \includegraphics[width=0.95\textwidth]{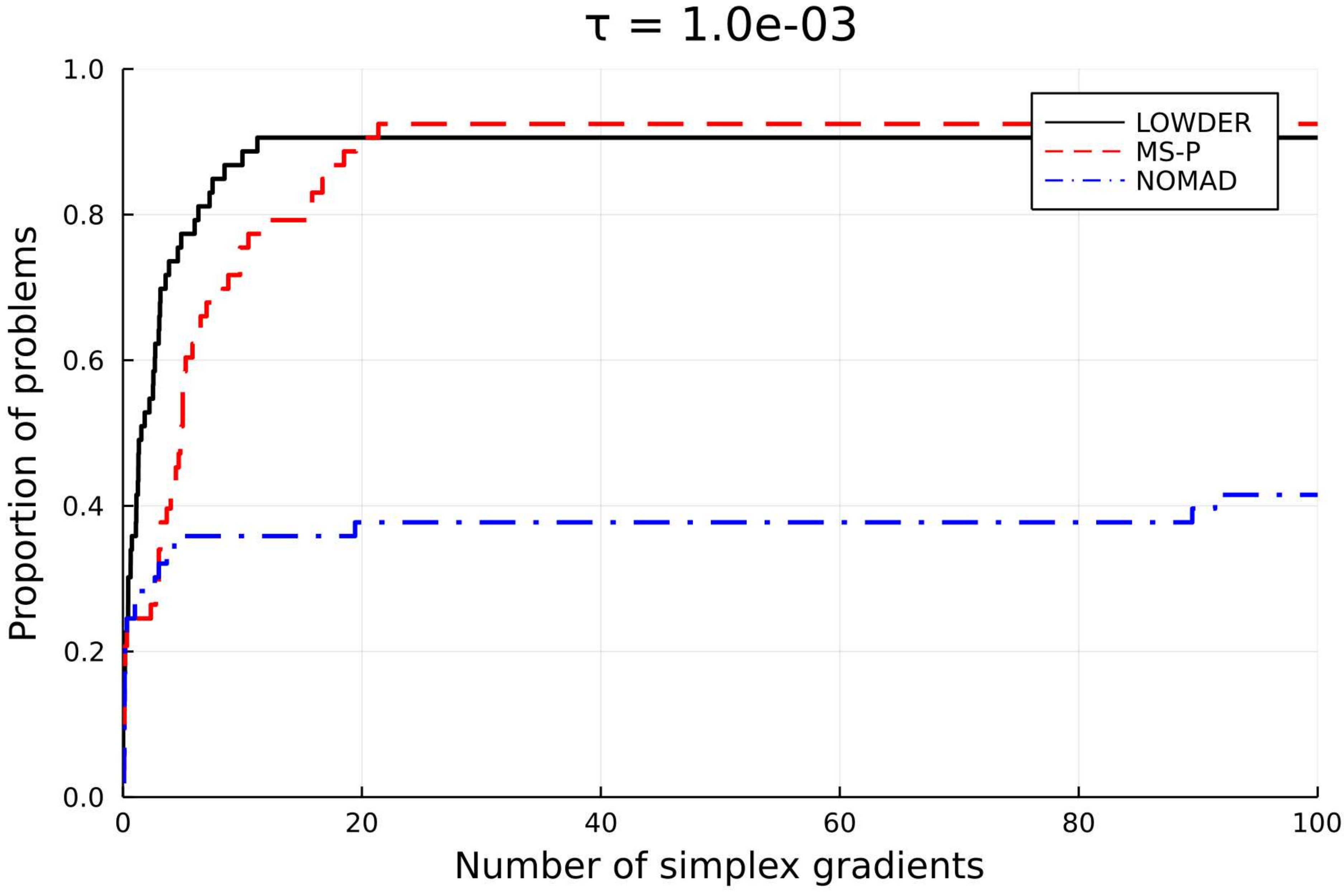}
    \end{minipage}
    \vskip\baselineskip
    \begin{minipage}{0.5\textwidth}
        \centering
        \includegraphics[width=0.95\textwidth]{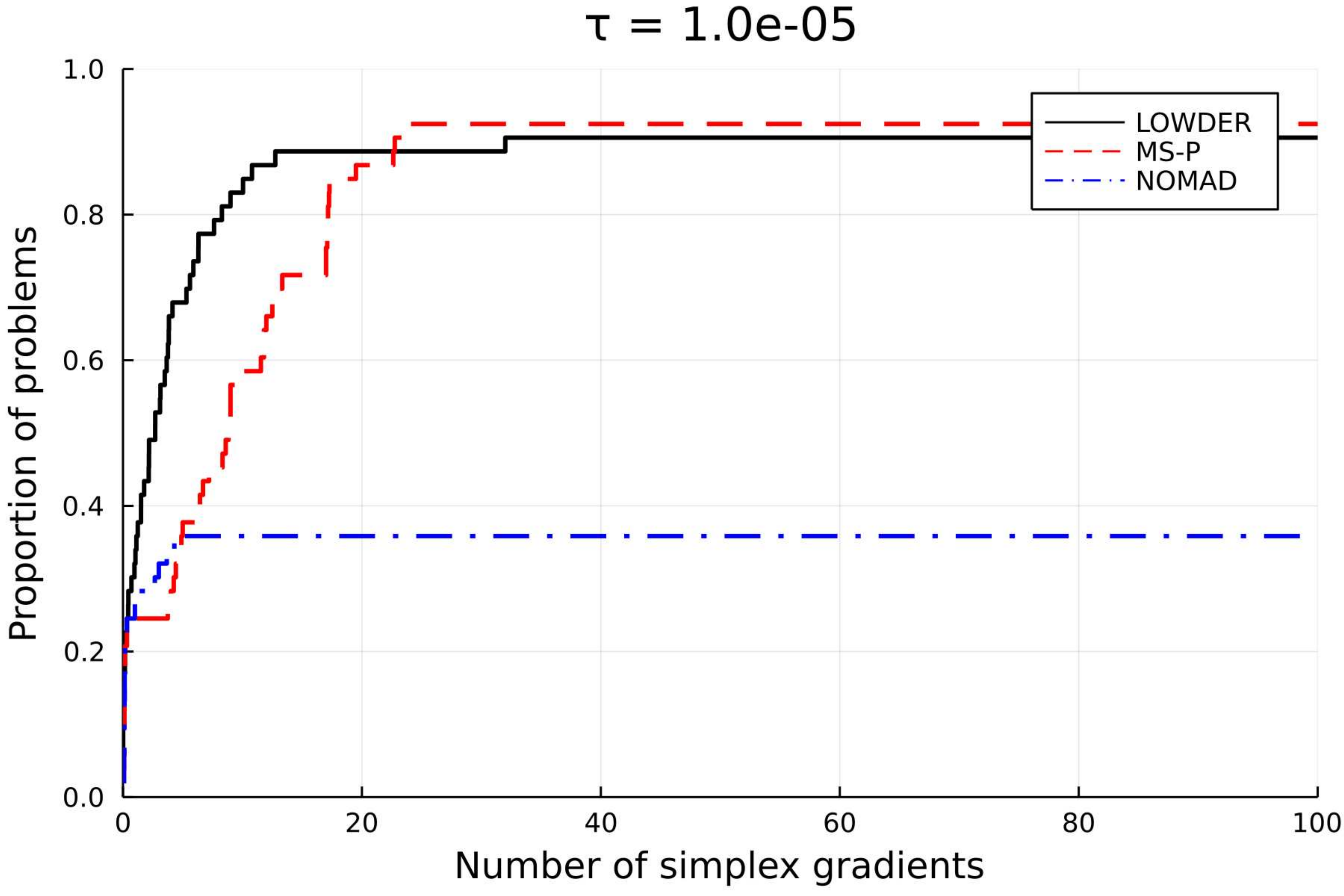}
    \end{minipage}\hfill
    \begin{minipage}{0.5\textwidth}
        \centering
        \includegraphics[width=0.95\textwidth]{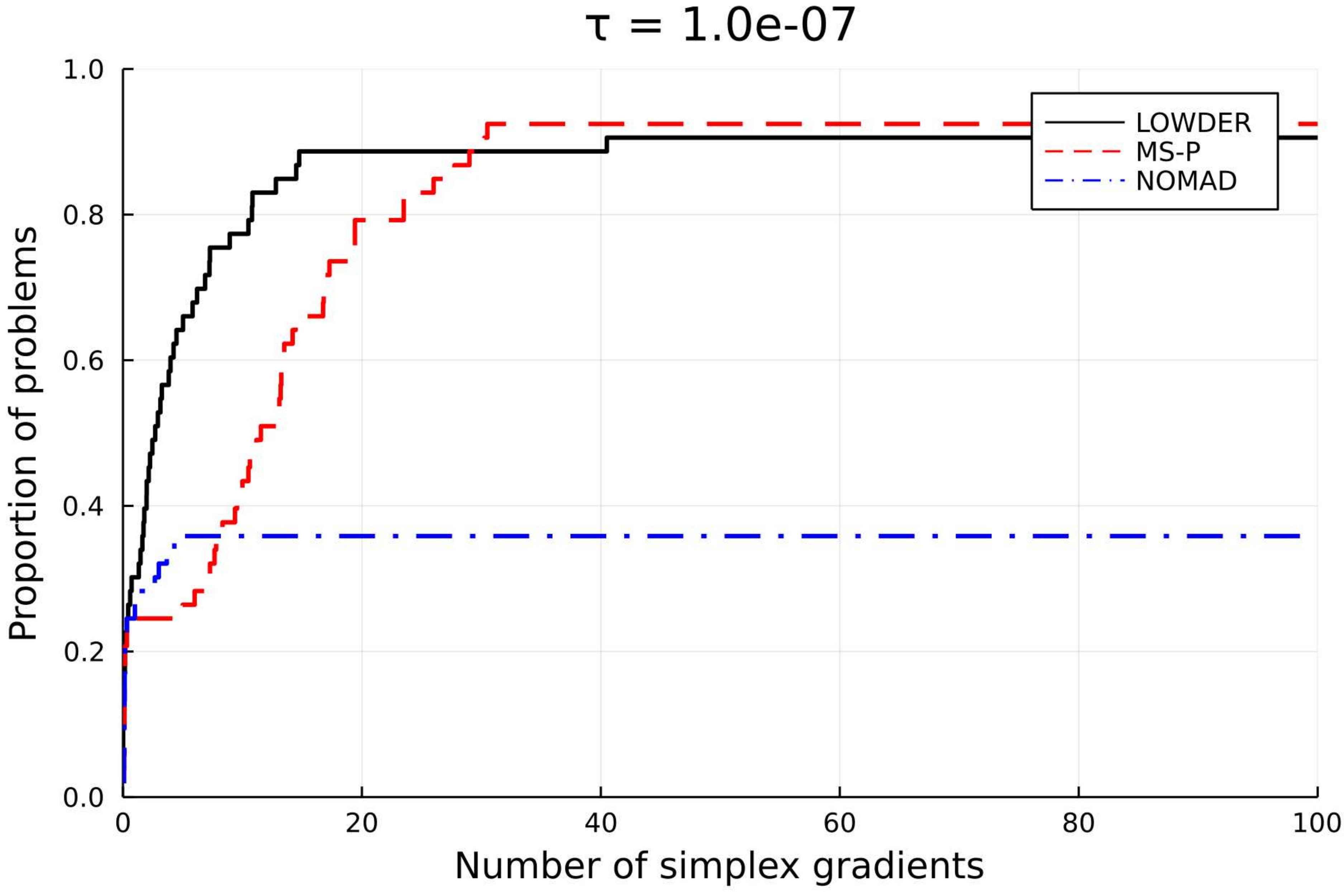}
    \end{minipage}
    \caption[Data profiles for the problems in MW test set.]{Data profiles for the problems in MW test set with tolerance $\tau \in \left\{10^{-1}, 10^{-3}, 10^{-5}, 10^{-7}\right\}$.}
    \label{imagem_data_profile_mw}
\end{figure}

\subsubsection{HS test set}

Analysing Figure \ref{imagem_data_profile_hs}, we can see that \texttt{NOMAD} can solve all problems for $\tau = 10^{-1}$ and can solve $95\%$ or more problems for other values of tolerance. \texttt{LOWDER} and \texttt{MS-P} have similar performances, alternating in the second position. When considering the computational budget of $100$ simplex gradients, \texttt{LOWDER} is overcome by \texttt{MS-P} by $1\%$ to $3\%$. For budget values between $20$ and $40$ simplex gradients, \texttt{LOWDER} performs slightly better than \texttt{MS-P}.

The great performance of \texttt{NOMAD} can be explained by the small dimension of the problems and by the recognized ability of the algorithm to solve complex problems, especially for nonlinear problems with constraints, as is the case of this test set. The use of linear models in \texttt{LOWDER} might have degraded its performance in the case of strongly nonlinear problems since the solutions generated by \texttt{TRSBOX} tend not to satisfy the conditions of the step acceptance phase (line \ref{alg:acceptance}). Therefore, set improvement iterations are applied, with directions calculated by \texttt{ALTMOV}, which may not be directions in which the objective function decreases. \texttt{NOMAD} also has two distinct mechanisms that help it escape from local minima. The first one is a global search procedure called \texttt{SEARCH} step, which can return any point on the underlying mesh, always trying to identify points that improve the best solution found so far. The second is a local search called the \texttt{POOL} step, and its purpose is to generate trial mesh points in the vicinity of the best solution \cite{ADMT2022}. Another behaviour of the \texttt{NOMAD} solver that we noticed during the execution of these problems was its tendency to find local and global solutions on the boundary of the feasible set. Among the $87$ problems that constitute the HS test set, \texttt{NOMAD} found boundary solutions for $55$ of them. Considering its good performance in the data profile, this suggests the solver is very efficient when exploring and evaluating the limits of the feasible set.

\begin{figure}[ht]
    \centering
    \begin{minipage}{0.5\textwidth}
        \centering
        \includegraphics[width=0.95\textwidth]{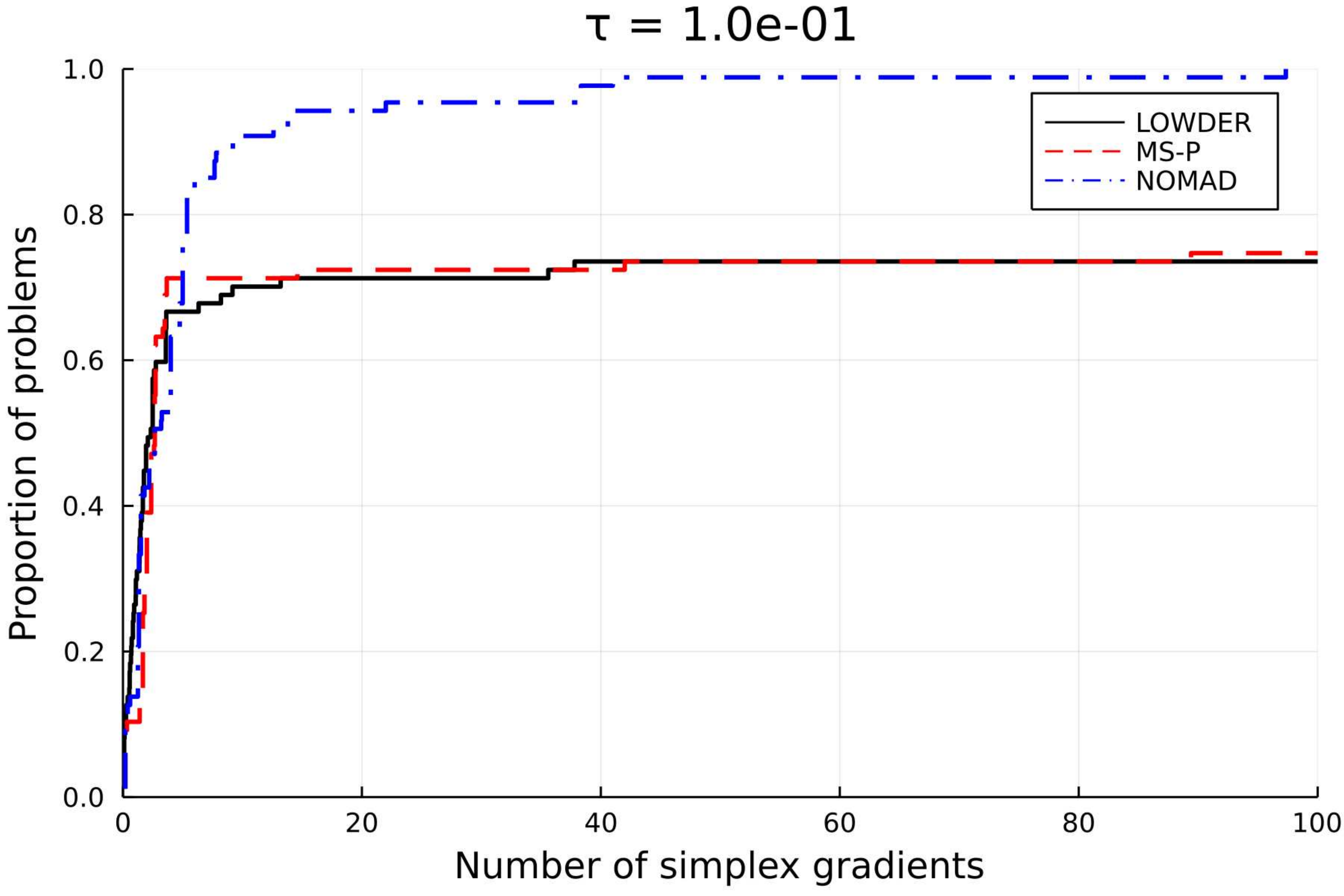}
    \end{minipage}\hfill
    \begin{minipage}{0.5\textwidth}
        \centering
        \includegraphics[width=0.95\textwidth]{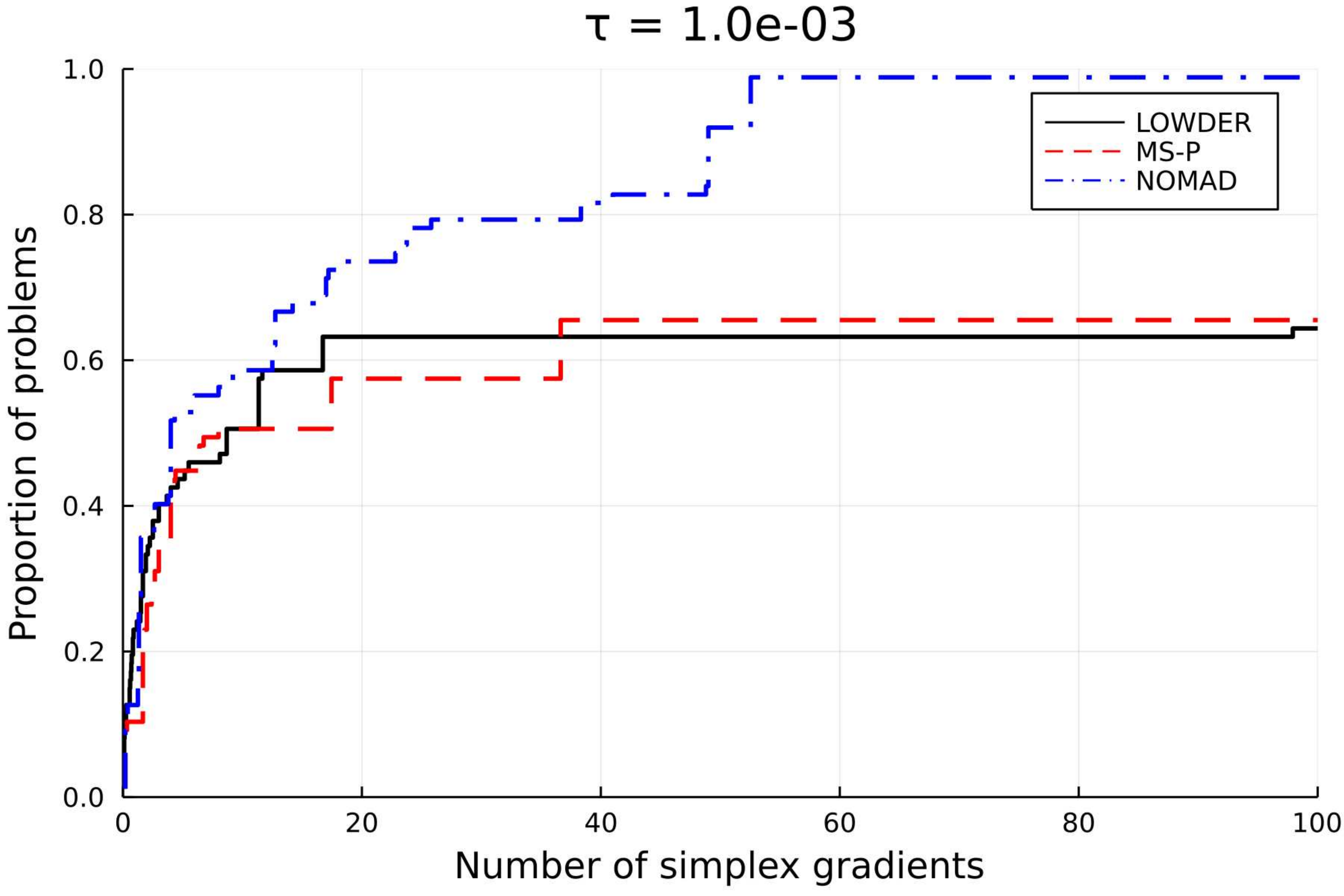}
    \end{minipage}
    \vskip\baselineskip
    \begin{minipage}{0.5\textwidth}
        \centering
        \includegraphics[width=0.95\textwidth]{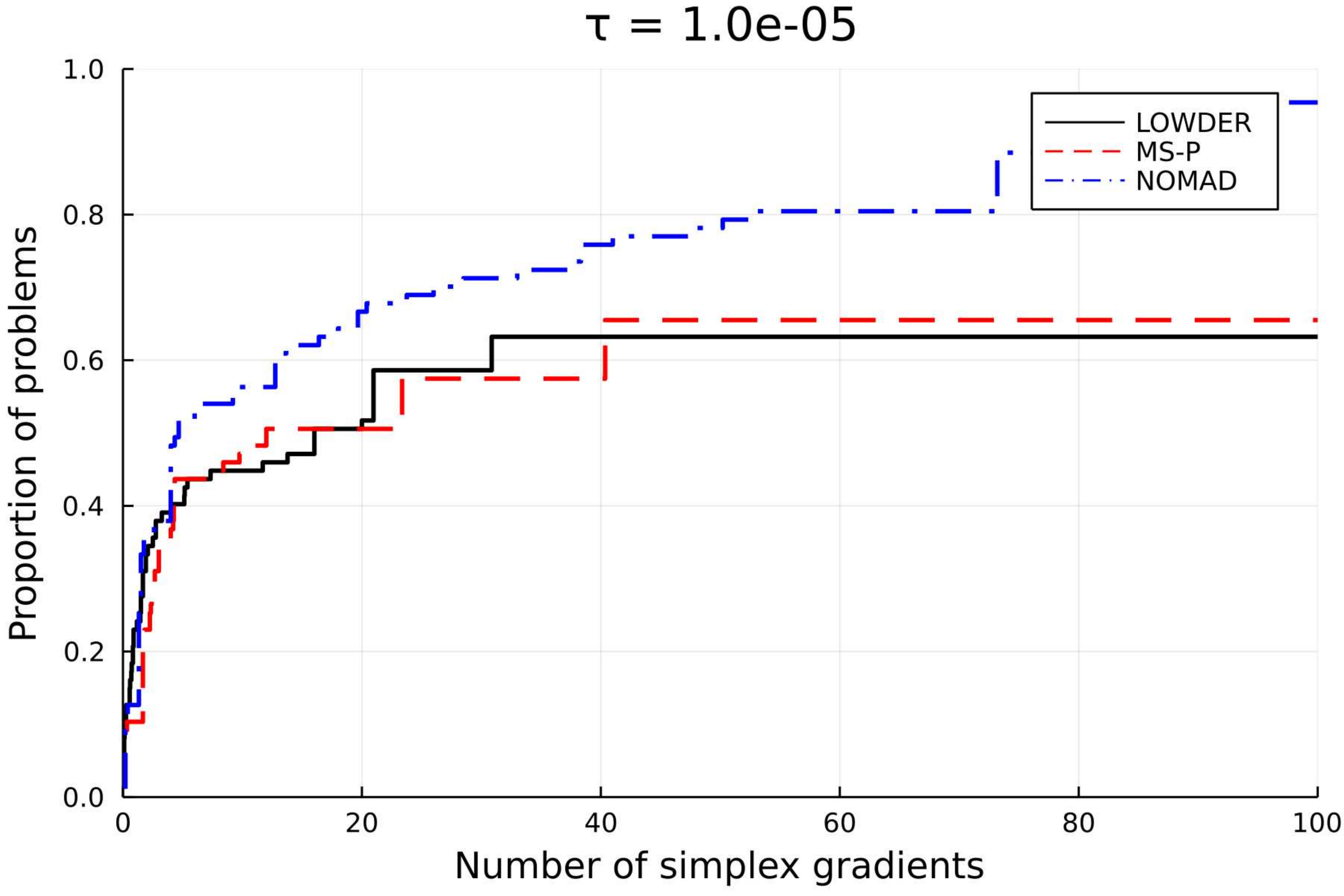}
    \end{minipage}\hfill
    \begin{minipage}{0.5\textwidth}
        \centering
        \includegraphics[width=0.95\textwidth]{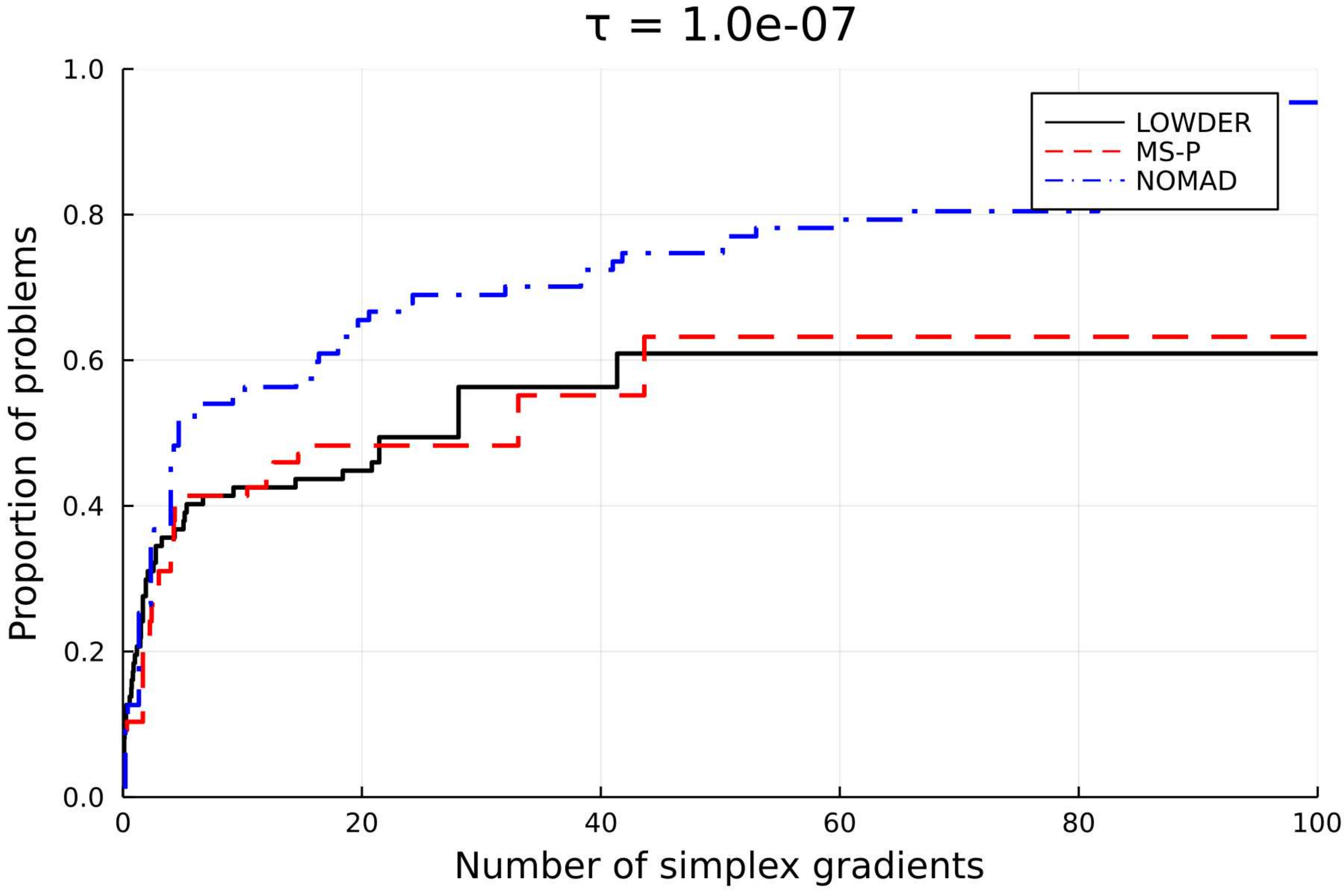}
    \end{minipage}
    \caption[Data profiles for the problems in HS test set.]{Data profiles for the problems in HS test set with tolerance $\tau \in \left\{10^{-1}, 10^{-3}, 10^{-5}, 10^{-7}\right\}$.}
    \label{imagem_data_profile_hs}
\end{figure}

When looking at the proportion of problems solved by the three solvers, we can see that they can solve about $40\%$ of problems with $5$ simplex gradients or less. From that point on, the performance differences are quite significant, especially when comparing \texttt{NOMAD} with the other solvers. By setting $\tau = 10^{-3}$, \texttt{NOMAD} can solve $60\%$ of problems with just $13$ simplex gradients, while \texttt{LOWDER} needs $17$, and \texttt{MS-P} needs about $37$ simplex gradients. Another situation worth mentioning occurs with $\tau = 10^{-7}$. In this case, considering the $50\%$ mark of problems solved, \texttt{NOMAD} can reach this value with $6$ simplex gradients, while \texttt{LOWDER} and \texttt{MS-P} need approximately $28$ and $33$ simplex gradients, respectively.

\subsubsection{QD test set}

\begin{figure}[ht]
    \centering
    \begin{minipage}{0.5\textwidth}
        \centering
        \includegraphics[width=0.95\textwidth]{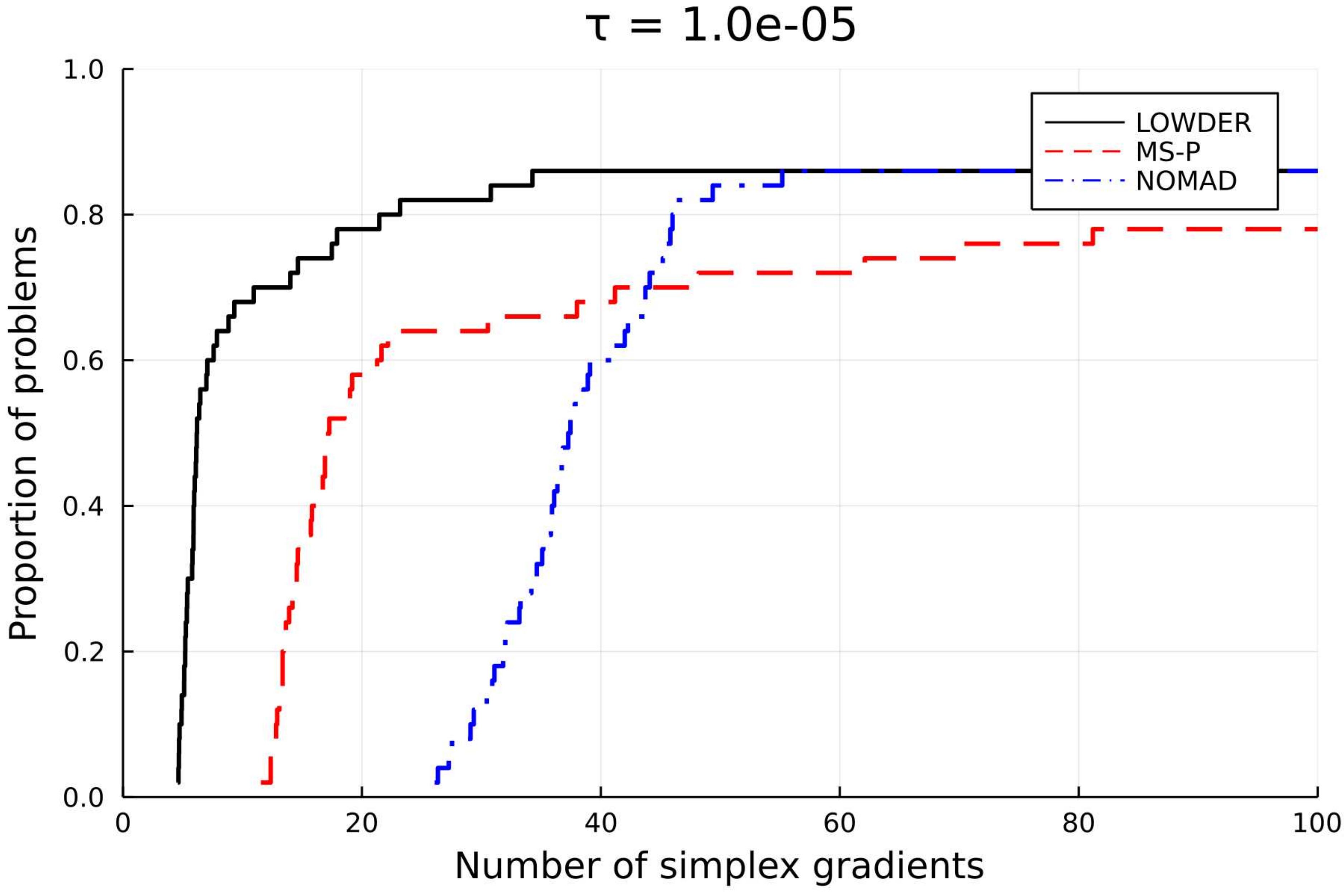}
    \end{minipage}\hfill
    \begin{minipage}{0.5\textwidth}
        \centering
        \includegraphics[width=0.95\textwidth]{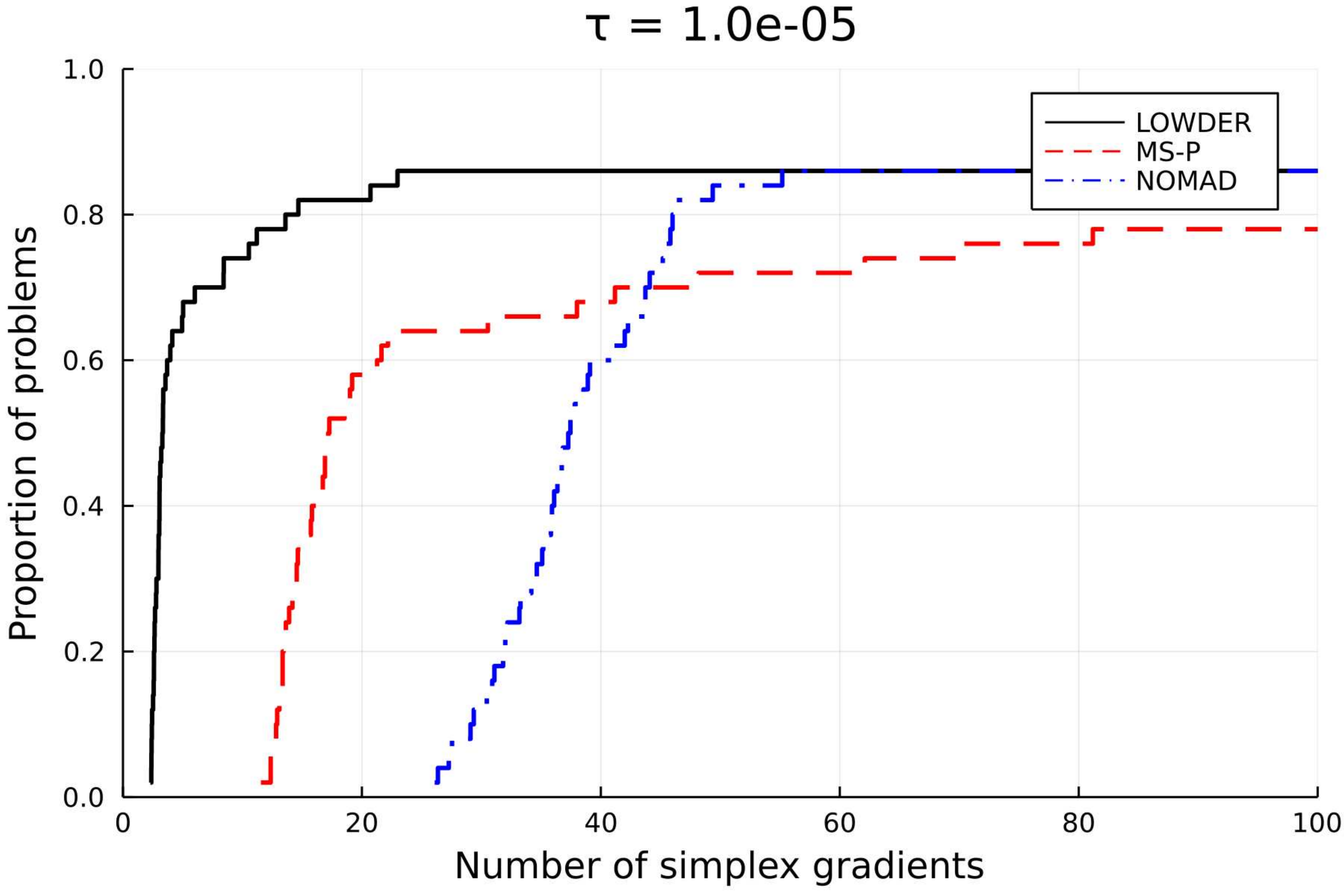}
    \end{minipage}
    \vskip\baselineskip
    \begin{minipage}{0.5\textwidth}
        \centering
        \includegraphics[width=0.95\textwidth]{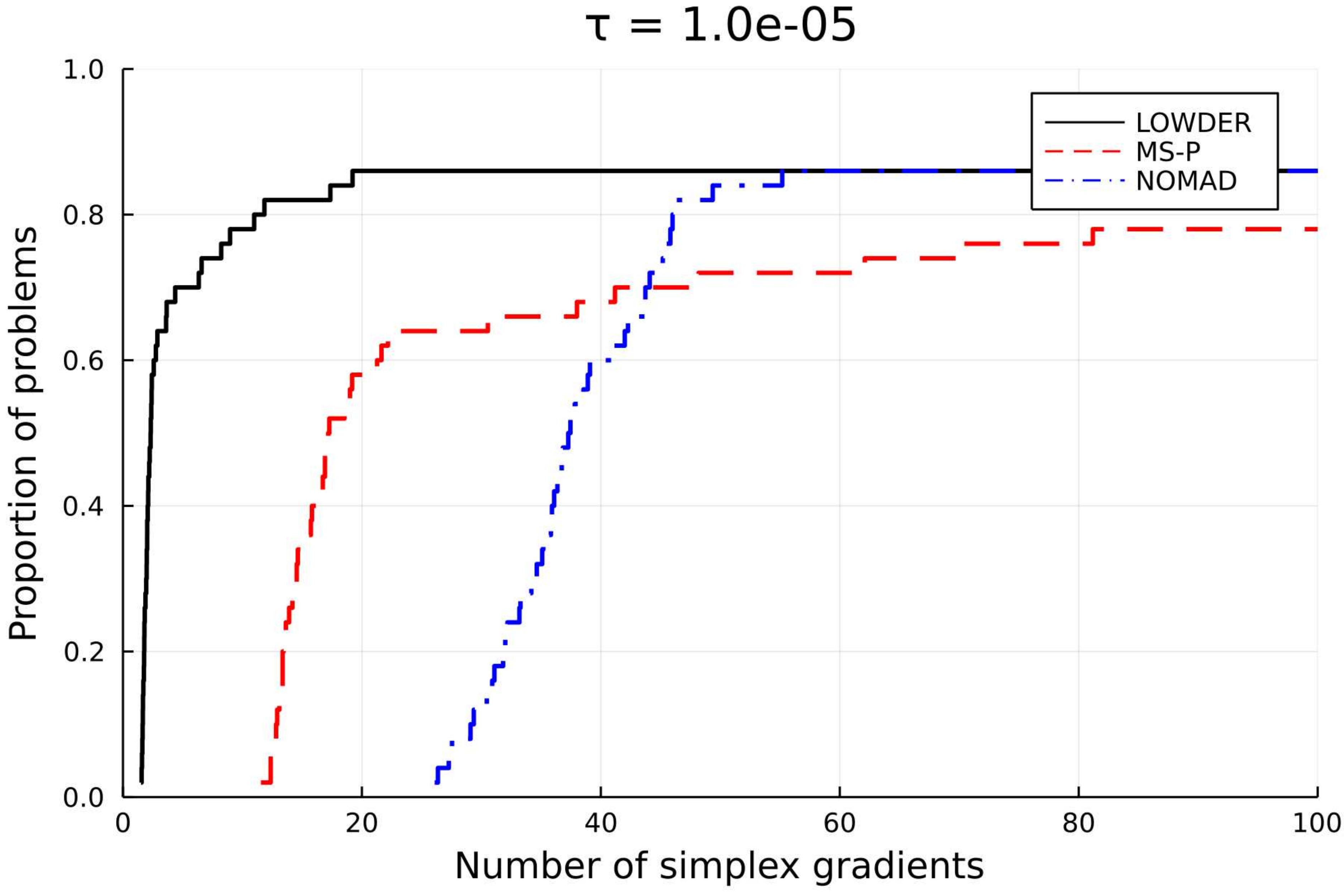}
    \end{minipage}\hfill
    \begin{minipage}{0.5\textwidth}
        \centering
        \includegraphics[width=0.95\textwidth]{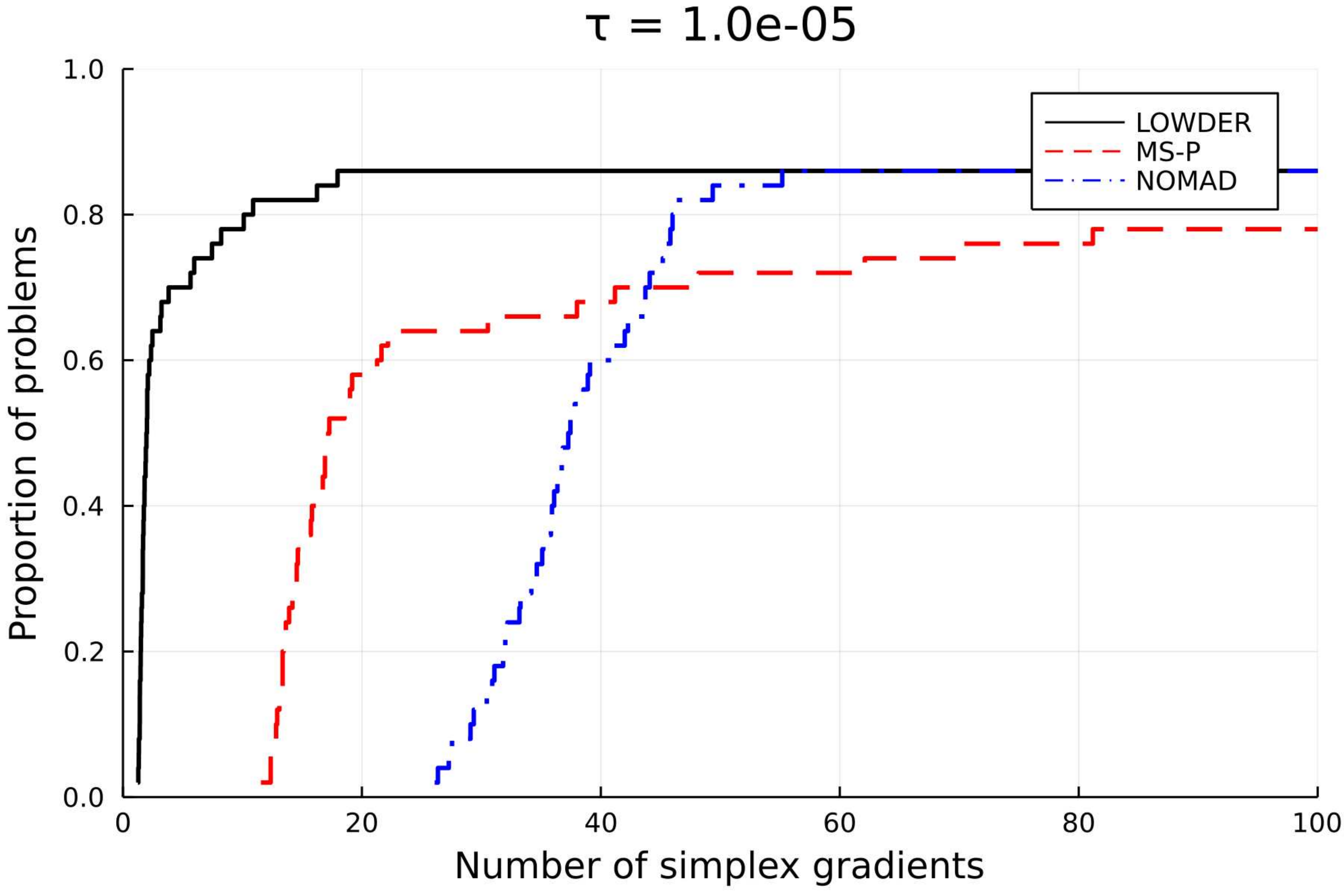}
    \end{minipage}
    \vskip\baselineskip
    \begin{minipage}{0.5\textwidth}
        \centering
        \includegraphics[width=0.95\textwidth]{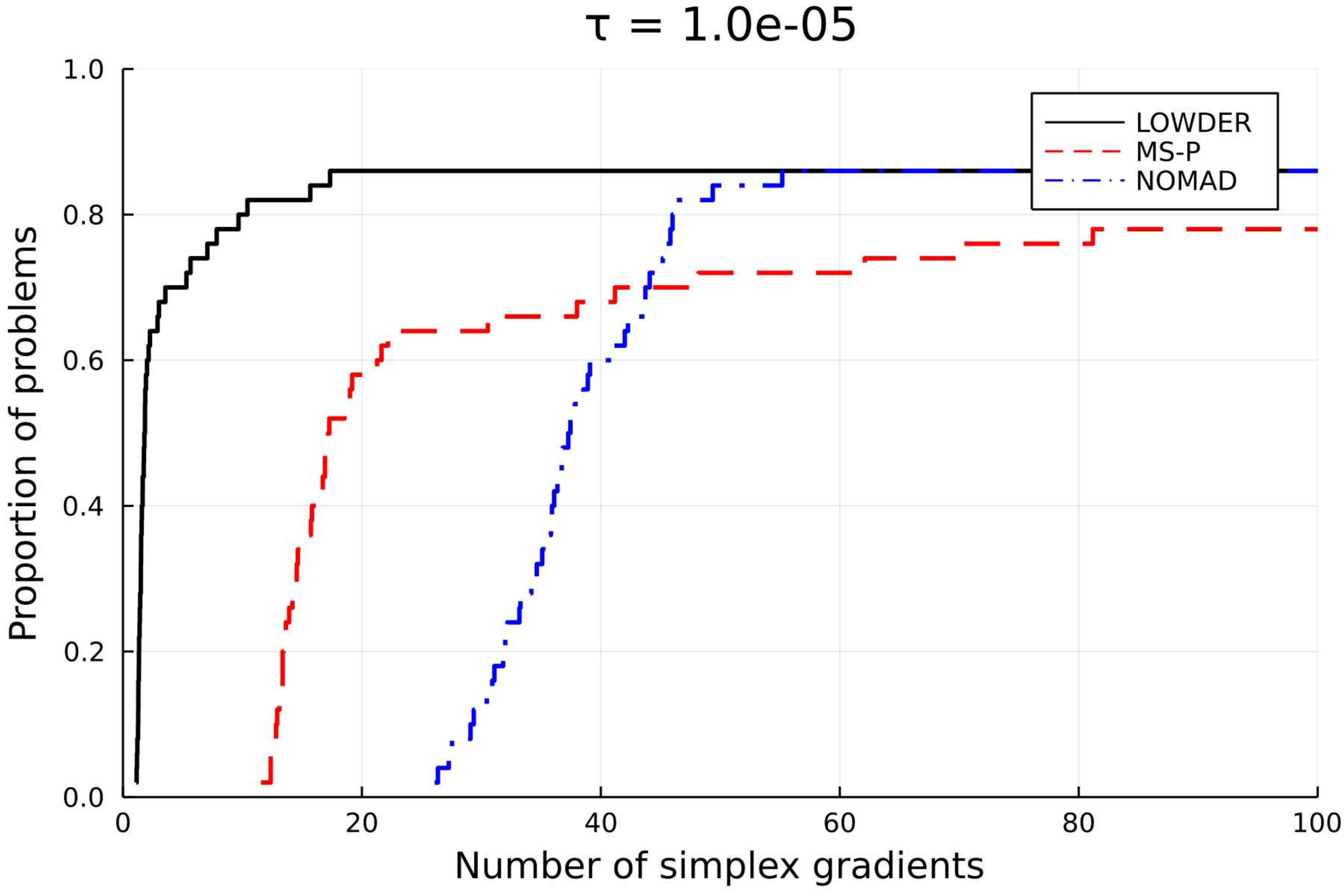}
    \end{minipage}
    \caption[Data profiles for the problems in QD test set.]{Data profiles for the problems in QD test set, with $\tau = 10^{-5}$, and $r \in \left\{ 10, 25, 50, 75, 100 \right\}$.}
    \label{imagem_data_profile_qd}
\end{figure}

Our goal in this test set is to show the benefits that \texttt{LOWDER} enjoys due to exploring the structure of LOVO problems. Figure \ref{imagem_data_profile_qd} presents the data profiles generated with tolerance $\tau = 10^{-5}$ for the problems in the QD test set, taking into account the subsets generated with $10$, $25$, $50$, $75$, and $100$ component functions. Analysing the plots presented, if we consider a computational budget of $100$ simplex gradients, the robustness of the algorithms is not affected by the number of functions $f_{i}$. \texttt{LOWDER} and \texttt{NOMAD} manage to solve approximately $85\%$ of the problems while \texttt{MS-P} solves $77\%$ of the problems, in the $5$ scenarios presented. Note that \texttt{LOWDER} outperforms \texttt{MS-P} and \texttt{NOMAD} by a large amount, especially for computational budgets of less than $40$ simplex gradients. For this fixed budget, \texttt{MS-P} also outperforms \texttt{NOMAD}, but the situation reverses for higher budget values. More specifically, the data profile shows us that for robustness rates above $70\%$, \texttt{NOMAD} is more efficient than \texttt{MS-P} and is able to match \texttt{LOWDER} using less than $60$ simplex gradients.

Another relevant fact is that the curves presented by \texttt{MS-P} and \texttt{NOMAD} do not appear to be affected by the variation in the number of component functions, while the \texttt{LOWDER} data profile curve has a clear tendency to approach the vertical axis according to the number of component functions increases. That is, the performance of \texttt{LOWDER} tends to improve, especially when considering small computational budgets, with less than $20$ simplex gradients. One way to visualize this influence is to verify the number of simplex gradients needed to solve a certain percentage of problems, as shown in Table \ref{tabela_lowder_qd}.

\begin{table}[ht]
    \centering
    \begin{tabular}{@{}cccccc@{}}
    \toprule
    Test subset & $20\%$ & $40\%$ & $60\%$ & $80\%$ & $85\%$ \\ \midrule
    QD10        & $5$  & $6$ & $7$  & $21$ & $34$ \\
    QD25        & $3$  & $3$ & $4$  & $14$ & $23$ \\
    QD50        & $2$  & $2$ & $3$  & $11$ & $19$ \\
    QD75        & $2$  & $2$ & $2$  & $10$ & $18$ \\
    QD100       & $1$  & $2$ & $2$  & $9$  & $17$ \\ \bottomrule
    \end{tabular}
    \caption[Performance of \texttt{LOWDER} in QD test sets.]{Approximate number of function simplex gradients that \texttt{LOWDER} needs to solve given ranges of problems on the QD test subsets.}
    \label{tabela_lowder_qd}
\end{table}

Note that as the number of component functions increases, represented by the number associated with the test subset, there is a tendency for the number of simplex gradients to decrease, and this can be verified for all ranges of problems considered.

\section{Conclusions}\label{sec:conclusions}

In this work, we introduced a new class of low order-value optimization methods, considering an approach reasoned on model-based derivative-free optimization. We presented a derivative-free trust-region algorithm for constrained black-box LOVO problems, which is based on the algorithms proposed by~\cite{CKPRS2013} and~\cite{VKPS2017}, and the ideas discussed by~\cite{AMM2008}. Our algorithm can deal with general closed convex constraints and is specially designed for problems whose values of functions $f_i$, $i = 1, \dots, r$, are provided through an oracle (black-box function). Note that we assume that we know how to project an arbitrary point onto the feasible set. Algorithm \ref{alg_lowder} has a structure very similar to the traditional trust-region framework and considers two different radii, one for the sample region and another for the trust-region. We allowed some freedom in the choice of models, as long as the gradient of the model is a good approximation of the gradient of the selected component function, in the sense of well poisedness (Assumption \ref{hip_grad_limitante}).

We discussed global convergence results and also worst-case iteration complexity, adopting common assumptions for this class of problems, as well as an interpretation from the perspective of the classical theory of LOVO problems. 
The analysis showed that the number of iterations and function evaluations performed by the algorithm is in line with what is expected for model-based methods, and that the adoption of minimum Frobenius norm quadratic models is competitive in terms of function evaluations when compared to completely determined models. 

We also presented \texttt{LOWDER}, our implementation of Algorithm \ref{alg_lowder} in \texttt{Julia} language, discussed implementation details and performed numerical tests. In its current version, \texttt{LOWDER} solves bound-constrained black-box LOVO problems and builds only determined linear models. \texttt{LOWDER} has several practical improvements, many of them derived from \texttt{BOBYQA}, a general-purpose derivative-free optimization solver for bound-constrained problems. In particular, \texttt{LOWDER} inherits the initial sampling mechanisms and has simplified versions of the \texttt{TRSBOX} and \texttt{ALTMOV} routines for solving the trust region subproblem (\ref{eq_def_subproblema}) and improving the geometry of the sample set, respectively. Like \texttt{BOBYQA}, we also avoided completely rebuilding models through an update mechanism. 

Since \texttt{LOWDER} is designed for LOVO problems, we compared it with algorithms that can handle, in some way, this type of problem. In this sense, we selected \texttt{MS-P} \cite{LM2022} and \texttt{NOMAD} \cite{ADMT2022}. \texttt{MS-P} is a manifold sampling algorithm for composition minimization problems. \texttt{NOMAD} is a well-established algorithm for black-box optimization problems based on the direct search. Three test sets for constrained and unconstrained LOVO problems were constructed. \texttt{LOWDER} had a behaviour similar to \texttt{MS-P} but, as the number of functions $r$ increased, its computational efficiency was shown.

Finally, it is well known that the Low Order-Value Optimization can generalize the nonlinear least-squares problem, as it allows us to discard observations considered outliers, as we can see in \cite{AMMY2009,CLSS2021}. Therefore, we can enhance \texttt{LOWDER} to take advantage of the structure of the least-squares problem, such as well-established algorithms like \texttt{DFO-GN} \cite{CR2019}, \texttt{POUNDERS} \cite{W2017}, and \texttt{DFBOLS} \cite{ZCS2010}, making it a competitive solver in this segment.




\backmatter

\bmhead{Data availability statement}

The algorithm proposed in this work is freely available at \url{https://github.com/aschwertner/LOWDER}. All the optimization problems used in the numerical experiments are also available at \url{https://github.com/aschwertner/LOWDER_Numerical_Tests}.

\bmhead{Statements and Declarations}

All authors certify that they have no affiliations with or involvement in any organization or entity with any financial interest or non-financial interest in the subject or materials discussed in this manuscript.

\bmhead{Acknowledgements}

We are grateful to Jeffrey Larson, co-author of the method \texttt{MS-P}, for providing the \texttt{MS-P} and \texttt{POUNDERS} codes and being available to help us.

\bibliography{library}

\end{document}